\documentclass[12pt]{article}
\usepackage{amsmath,amssymb,amsthm}
\numberwithin{equation}{section}
\newcommand{\g}{\mathfrak{g}}
\newcommand{\mr}{\mathfrak{r}}
\newcommand{\p}{\varphi}

\newcommand{\pe}{\mathfrak{p}}
\newcommand{\q}{\mathbf{Q}}

\newcommand{\lw}{L \cap W}
\newcommand{\ord}{\mathrm{ord}_{\mathfrak{p}}}

\newcommand{\REF}[1]{\eqref{#1}}
\newcommand{\THM}[1]{Theorem \ref{#1}}
\newcommand{\COR}[1]{Corollary \ref{#1}}
\newcommand{\LEM}[1]{Lemma \ref{#1}}
\newcommand{\PROP}[1]{Proposition \ref{#1}}

\title{Invariants and discriminant ideals of 
	orthogonal complements in a quadratic space} 
\author{Manabu Murata} 
\date{}
\theoremstyle{definition}
\newtheorem{sect}{}[section]

\theoremstyle{plain}
\newtheorem{thm}[sect]{Theorem}
\newtheorem{lem}[sect]{Lemma}
\newtheorem{prop}[sect]{Proposition}
\newtheorem{cor}[sect]{Corollary}

\begin{document}
\maketitle
\begin{abstract}
This paper studies two topics concerning the orthogonal complement $W$ of $Fh$ 
with respect to a quadratic form $\p$ on a vector space $V$ over a number field $F$ 
for a given element $h$ of $V$. 
One is to determine the invariants of the isomorphism class of $W$ in the sense of Shimura \cite{cq}. 
The other is to study a fractional ideal $[M/L \cap W]$ in $F$ that is 
closely connected with the ideal $\p(h,\, L)$, 
where $L$ is a maximal lattice in $V$ and $M$ is a maximal lattice in $W$. 
We also discuss the class number of the genus of maximal lattices in $W$ 
when $V$ is even-dimensional and $L \cap W$ is maximal. 
\end{abstract}

\section*{Introduction}

Let $V$ be a vector space of dimension $n\,(>1)$ over an algebraic number field $F$ 
and $\p$ a nondegenerate symmetric $F$-bilinear form on $V$. 
For an element $h$ of $V$ such that $\p(h,\, h) \ne 0$, 
we consider a subspace 
\begin{gather*}
W = (Fh)^{\perp} = \{x \in V \mid \p(x,\, h) = 0\}
\end{gather*}
of dimension $n - 1$ and the restriction $\psi$ of $\p$ to $W$. 

One of purposes in this paper is to study the \textit{invariants} of 
the quadratic space $(W,\, \psi)$ in the sense of Shimura \cite{cq}, 
which are given by a set of data that consists of 
the \textit{dimension} of $W$ over $F$, 
the \textit{discriminant field} of $\psi$, 
the \textit{characteristic quaternion algebra} of $\psi$, 
and the \textit{index} of $\psi$ at each real archimedean prime of $F$; 
see Section 1.1 in the text for details. 
It is known by \cite[Theorem 4.2]{cq} that 
these invariants determine the isomorphism class of $(W,\, \psi)$, 
and vice versa. 
We shall give such invariants in terms of 
the invariants of $(V,\, \p)$ and the value $\p[h] = \p(h,\, h)$ in \THM{t1}. 
\\

To explain another purpose, 
take $h \in V$ and put $W = (Fh)^{\perp}$ as above. 
For two $\g$-lattices $M$ and $N$ in $W$ 
we denote by $[M/N]$ the $\g$-ideal of $F$ generated over $\g$ by 
$\det(\alpha)$ of all $F$-linear automorphisms $\alpha$ of $W$ such that $M\alpha \subset N$. 
Here $\g$ is the ring of all algebraic integers in $F$. 
Let $L$ be a $\g$-maximal lattice in $V$ with respect to $\p$. 
Then we have a $\g$-lattice $L \cap W$ in $W$ such that $\psi[x]\in\g$ for every $x\in\lw$, 
which is our main object in this paper. 
Here the term \textit{maximal} is given in Section 2.1. 
Suppose $M$ is $\g$-maximal with respect to $\psi$. 
The ideal $[\widetilde{M}/M]$ is called the \textit{discriminant ideal} of $(W,\, \psi)$, 
where we put $\widetilde{M} = \{x \in W \mid 2\psi(x, M) \subset \g\}$. 
Such an ideal is independent of the choice of $M$, 
and it was introduced in \cite[{\S}6.1]{cq}. 

It can be seen that $[(L \cap W)\widetilde{\,}/L \cap W] = [M/L \cap W]^{2}[\widetilde{M}/M]$ 
and in particular, 
$\lw$ is maximal if $[M/L \cap W] = \g$. 
Now, 
put $q = \p[h] \ne 0$ and observe that $\p(h,\, L)$ defines a $\g$-ideal of $F$. 
Then our purpose is to show in \THM{t3} that 
\textit{there exists a $\g$-ideal $\mathfrak{b}(q)$ of $F$ such that} 
\begin{gather}
\p(h,\, L)[M/L \cap W] = 2^{-1}\mathfrak{b}(q), \label{mr}
\end{gather}
\textit{and it is determined by the relation} 
\begin{gather*}
\mathfrak{b}(q)^{2}[\widetilde{M}/M] = 2q[\widetilde{L}/L]. 
\end{gather*}
The ideal $\mathfrak{b}(q)$ is independent of the choice of $h$ and $L$. 
From \REF{mr} we know that $L \cap W$ is maximal in $W$ if and only if $\p(h,\, L) = 2^{-1}\mathfrak{b}(q)$, 
that is, if $q(\p(h,\, L))^{-2} = 2[\widetilde{M}/M][\widetilde{L}/L]^{-1}$. 
This gives the criterion on the maximality of $\lw$ due to Yoshinaga \cite[Theorem 6.3]{Y} 
in terms of the discriminant ideals of $\p$ and $\psi$. 

As an application, 
we shall discuss how the class number of the genus 
of all maximal lattices in a quadratic space can be determined 
through the principle in \cite[Theorem 11.6]{04} due to Shimura. 
In \PROP{cnf} we treat \textit{even} dimensional quadratic spaces whose discriminant fields are the base fields. 
As for the case where the discriminant field is not the base field, 
we take up the quadratic form over $\q$ defined by the sum of six squares in the last section. 
\\

\textit{Notation}. 
We denote by $\mathbf{Z}$, $\mathbf{Q}$, and $\mathbf{R}$ the ring of rational integers, 
the field of rational numbers, and the field of real numbers, respectively. 

If $R$ is an associative ring with identity element and if $M$ is an $R$-module, 
then we write $R^{\times}$ for the group of all invertible elements of $R$ and 
$M_{n}^{m}$ the $R$-module of $m \times n$-matrices with entries in $M$. 
We set $R^{\times 2} = \{a^{2} \mid a \in R^{\times}\}$. 
For a finite set $X$, 
we denote by $\# X$ the number of elements in $X$. 
If a set $X$ is a disjoint union of its subsets $Y_{1},\, \cdots ,\, Y_{m}$, 
then we write $X = \bigsqcup_{i = 1}^{m} Y_{i}$. 
We also write $\mathrm{diag}[a_{1},\, \cdots ,\, a_{s}]$ 
for the matrix with square matrices $a_{1},\, \cdots ,\, a_{s}$ in the diagonal blocks 
and $0$ in all other blocks. 
We set $[a] = \mathrm{Max}\{n \in \mathbf{Z} \mid n \le a\}$ 
and $\delta_{ij} = 1$ or $0$ according as $i = j$ or $i \ne j$. 

Let $V$ be a vector space over a field $F$ of characteristic $0$, 
and $GL(V)$ the group of all $F$-linear automorphisms of $V$. 
We let $GL(V)$ act on $V$ on the \textit{right}. 

Let $F$ be an algebraic number field (of finite degree) and $\g$ the ring of all algebraic integers in $F$. 
For a fractional ideal in $F$ we often call it a $\g$-ideal. 
Let $\mathbf{a}$, $\mathbf{h}$, and $\mathbf{r}$ 
be the sets of archimedean primes, nonarchimedean primes, and real archimedean primes of $F$, 
respectively. 
We denote by $F_{v}$ the completion of $F$ at $v \in \mathbf{a} \cup \mathbf{h}$ 
and by $F_{\mathbf{A}}$ the adele ring of $F$. 
We often identify $v$ with the prime ideal of $F$ corresponding to $v \in \mathbf{h}$, 
and write $x_{v}$ for the image of $x$ under the embedding of $F$ into $\mathbf{R}$ at $v \in \mathbf{r}$. 
For $v \in \mathbf{h}$, 
we denote by $\g_{v}$, $\mathfrak{p}_{v}$, and $\pi_{v}$ 
the maximal order of $F_{v}$, the prime ideal in $F_{v}$, and a prime element of $F_{v}$, 
respectively. 
We write $\mathrm{ord}_{\mathfrak{p}}(a) = \mathrm{ord}_{v}(a) = m$ 
if $a\g = \mathfrak{p}^{m}$ for $0 \ne a \in F$, 
where $\pe$ is the prime ideal of $F$ corresponding to $v$. 
If $K$ is a quadratic extension of $F$, 
we denote by $D_{K/F}$ the relative discriminant of $K$ over $F$, 
and put $K_{v} = K \otimes_{F} F_{v}$ for $v \in \mathbf{h}$. 
For $b \in F_{v}^{\times}$ we set 
\begin{gather*}
\xi_{v}(b) = 
	\begin{cases}
	1 & \text{if $b \in F_{v}^{\times 2}$}, \\
	-1 & \text{if $F_{v}(\sqrt{b})$ is an unramified quadratic extension of $F_{v}$}, \\
	0 & \text{if $F_{v}(\sqrt{b})$ is a ramified quadratic extension of $F_{v}$}. 
	\end{cases}
\end{gather*}
By a \textit{$\g$-lattice} $L$ in a vector space $V$ over a number field or nonarchimedean local field $F$, 
we mean a finitely generated $\g$-submodule in $V$ containing a basis of $V$. 
For two subspaces $X$ and $Y$ of $V$, 
we denote by $X \oplus Y$ the direct sum of $X$ and $Y$ if $\p(x,\, y) = 0$ for every $x \in X$ and $y \in Y$; 
we also denote by $\p|_{X}$ the restriction of $\p$ to $X$. 
When $X$ is an object defined over a number field $F$, 
we often denote by $X_{v}$ the localization at a prime $v$ if it is meaningful. 

\section{Invariants of an orthogonal complement}

\subsection{Quadratic spaces and invariants}

Let $F$ be an algebraic number field or nonarchimedean local field. 
Let $(V,\, \p)$ be a quadratic space over $F$, 
that is, 
$V$ is a vector space of dimension $n$ over $F$ 
and $\p$ is a symmetric $F$-bilinear form on $V$. 
In this paper, 
we consider only a nondegenerate form $\p$. 
We put $\p[x] = \p(x,\, x)$ for $x \in V$. 
We define the orthogonal group and the special orthogonal group of $\p$ by 
\begin{gather*}
O^{\p}(V) = O^{\p} = \{ \gamma \in GL(V) \mid \p(x\gamma,\, y\gamma) = \p(x,\, y) 
	\text{ for every } x,\, y\in V\}, \\ 
SO^{\p}(V) = SO^{\p} = \{ \gamma \in O^{\p}(V) \mid \det(\gamma) = 1 \}. 
\end{gather*}
We denote by $A(\p)$ (or by $A(V)$) the Clifford algebra of $\p$ and 
by $A^{+}(\p)$ (or by $A^{+}(V)$) the even Clifford algebra of $\p$. 
\\ 

For a number field $F$ and $v \in \mathbf{a} \cup \mathbf{h}$, 
we put $V_{v} = V\otimes_{F}F_{v}$ and denote by $\p_{v}$ the $F_{v}$-bilinear extension of $\p$ to $V_{v}$; 
we set $(V,\, \p)_{v} = (V_{v},\, \p_{v})$. 
For $v \in \mathbf{h}$, 
$(V,\, \p)_{v}$ has a \textit{Witt decomposition} as follows (cf. \cite[Lemma 1.3]{04}): 
There exist $2r_{v}$ elements $e_{i}$ and $f_{i}$ $(i = 1,\, \cdots,\, r_{v})$ such that 
\begin{gather}
V_{v} = Z_{v} \oplus \sum_{i=1}^{r_{v}} (F_{v}e_{i} + F_{v}f_{i}), \label{w1} \\ 
Z_{v} = \{ z \in V_{v} \mid \p_{v}(z,\, e_{i}) = \p_{v}(z,\, f_{i}) = 0 \text{ for every $i$} \}, \nonumber \\ 
\p_{v}(e_{i},\, e_{j}) = \p_{v}(f_{i},\, f_{j}) = 0, \quad 
2\p_{v}(e_{i},\, f_{j}) = \delta_{ij}. \nonumber 
\end{gather}
Here the restriction of $\p_{v}$ to $Z_{v}$ is anisotropic. 
It is known that the dimension $t_{v}$ of $Z_{v}$ is determined by $\p$ and $v$, 
and that $0 \le t_{v} \le 4$ for $v \in \mathbf{h}$ (cf.\ \cite[Theorem 7.6\ (ii)]{04}). 
We call $Z_{v}$ a \textit{core subspace} of $(V,\, \p)_{v}$ 
and $t_{v}$ the \textit{core dimension} of $(V,\, \p)$ at $v$. 
For convenience, 
we also call a subspace $U_{v}$ of $V_{v}$ anisotropic if $\p_{v}$ is so on $U_{v}$. 
\\

By the \textit{invariants} of $(V,\, \p)$, 
we understand a set of data 
\begin{gather*}
\left\{n,\ F(\sqrt{\delta}),\ Q(\p),\ \{s_{v}(\p)\}_{v \in \mathbf{r}}\right\}, 
\end{gather*}
where $n$ is the dimension of $V$ over $F$, 
$F(\sqrt{\delta})$ is the \textit{discriminant field} of $\p$ with $\delta = (-1)^{n(n-1)/2}\det(\p)$, 
$Q(\p)$ is the \textit{characteristic quaternion algebra} of $\p$, 
and $s_{v}(\p)$ is the \textit{index} of $\p$ at $v \in \mathbf{r}$. 
For these definitions, 
the reader is referred to \cite[{\S}1.1,\ {\S}3.1,\ and {\S}4.1]{cq} (see also below). 
By virtue of \cite[Theorem 4.2]{cq} the isomorphism class of $(V,\, \p)$ is determined by 
$\{n,\, F(\sqrt{\delta}),\, Q(\p),\, \{s_{v}(\p)\}_{v \in \mathbf{r}}\}$. 
By the definition of $Q(\p)$, 
\begin{gather*}
A(\p) \cong M_{s}(Q(\p)) \quad \text{if $n$ is even and $n>0$}, \\
A^{+}(\p) \cong M_{s}(Q(\p)) \quad \text{if $n$ is odd and $n>1$} 
\end{gather*}
as a central simple algebra over $F$ with $0 < s \in \mathbf{Z}$. 
We set $Q(\p) = M_{2}(F)$ if $n = 1$. 
For $v \in \mathbf{r}$, 
$\p$ can be represented by a matrix of the form 
$\mathrm{diag}[1_{i_{v}},\, -1_{j_{v}}]$ with $i_{v} + j_{v} = n$, 
where $1_{m}$ is the identity element of $\mathbf{R}_{m}^{m}$; 
then we define the index $s_{v}(\p)$ of $\p$ by $s_{v}(\p) = i_{v} - j_{v}$. 

The characteristic algebra $Q(\p_{v})$ is also defined for $\p_{v}$ at $v \in \mathbf{a} \cup \mathbf{h}$ 
and it coincides with $Q(\p)\otimes_{F}F_{v}$ (cf.\ \cite[{\S}3.1]{cq}). 
By \cite[Lemma 3.3]{cq} the isomorphism class of $(V,\, \p)_{v}$ is determined by 
$\{n,\, F_{v}(\sqrt{\delta}),\, Q(\p_{v})\}$ if $v \in \mathbf{h}$. 
As for $v \in \mathbf{a}$, 
it is determined by $\{n,\, s_{v}(\p)\}$ if $v \in \mathbf{r}$, 
and by the dimension $n$ if $v \not\in \mathbf{r}$. 
If $v \in \mathbf{r}$, 
then $Q(\p_{v})$ is given by 
\begin{gather}
Q(\p_{v}) = 
	\begin{cases}
	M_{2}(\mathbf{R}) & \text{if } s_{v}(\p) \equiv \pm 1,\ 0,\ 2 \pmod{8}, \\
	\mathbf{H} & \text{if } s_{v}(\p) \equiv \pm 3,\ 4,\ 6 \pmod{8}, 
	\end{cases} \label{chre}
\end{gather}
where $\mathbf{H}$ is the division ring of Hamilton quaternions; 
see \cite[(4.2a) and (4.2b)]{cq}, for example. 
If $v \not\in \mathbf{r}$, 
then $Q(\p_{v}) = M_{2}(\mathbf{C})$, 
where $\mathbf{C}$ is the field of complex numbers. 
It can be seen that 
$Q(\p_{v})$ coincides with the characteristic algebra of 
a core subspace of $(V,\, \p)_{v}$ for $v \in \mathbf{h}$. 
\\

As was shown in \cite[{\S}3.2]{cq} and also in the proof of \cite[Lemma 3.3]{cq}, 
the core dimension $t_{v}$ of $(V,\, \p)$ at $v \in \mathbf{h}$ is determined as follows: 
\begin{enumerate}
\item If $n$ is even, 
then 
\begin{gather}
t_{v} = 
\begin{cases}
0 & \text{if $F_{v}(\sqrt{\delta}) = F_{v}$ and $Q(\p_{v}) = M_{2}(F_{v})$}, \\
4 & \text{if $F_{v}(\sqrt{\delta}) = F_{v}$ and $Q(\p_{v})$ is a division algebra}, \\
2 & \text{if $F_{v}(\sqrt{\delta}) \ne F_{v}$},
\end{cases} \label{cd1}
\end{gather}
where $F(\sqrt{\delta})$ is the discriminant field of $\p$. 
\item If $n$ is odd, 
then 
\begin{gather}
t_{v} = 
\begin{cases}
1 & \text{if $Q(\p_{v}) = M_{2}(F_{v})$}, \\
3 & \text{if $Q(\p_{v})$ is a division algebra}.
\end{cases} \label{cd2}
\end{gather}
\end{enumerate}

Let $F$ be a number field or nonarchimedean local field. 
For $h \in V$ such that $\p[h] \ne 0$, 
we put 
\begin{gather}
W = (Fh)^{\perp} = \{x \in V \mid \p(x,\, h) = 0\}. \label{orc}
\end{gather}
Then $V = W \oplus Fh$. 
Since $\p[x] \in F^{\times 2}\p[h]$ if $x \in Fh$, 
we often denote one-dimensional $(Fh,\, \p|_{Fh})$ by $\langle F,\, q \rangle$ with $q = \p[h]$, 
and then $(V,\, \p)$ by $(V,\, \p) = (W,\, \psi) \oplus \langle F,\, q \rangle$ 
with the restriction $\psi$ of $\p$ to $W$. 
It is noted that the invariants of $(W,\, \psi)$ are independent of the choice of $h$. 
To see this, 
let $h_{0}$ be another element of $V$ so that $\p[h_{0}] = q$. 
Then there exists $\gamma \in SO^{\p}(V)$ such that $h\gamma = h_{0}$ by virtue of \cite[Lemma 1.5(ii)]{04}. 
Thus $(W,\, \psi)$ is isomorphic to $(W_{0},\, \psi_{0})$ under $\gamma$ 
and its invariants are the same as those of $(W_{0},\, \psi_{0})$, 
where $W_{0} = (Fh_{0})^{\perp}$ and $\psi_{0}$ is the restriction of $\p$ to $W_{0}$. 
This shows our claim. 
Notice that the isomorphism class of $(W,\, \psi)$ is determined by $q \in F^{\times}/F^{\times 2}$. 
\\

Here we introduce some symbols for lower-dimensional quadratic spaces, 
which will be used throughout the paper. 

For a quadratic extension field $K$ of $F$, 
we put $2\kappa(x,\, y) = xy^{\rho} + x^{\rho}y$ for $x,\, y \in K$ 
with a nontrivial automorphism $\rho$ of $K$ over $F$. 
Put also $\kappa[x] = \kappa(x,\, x)$, 
which is the norm $N_{K/F}(x)$ of $x$. 

For a quaternion algebra $B$ over $F$, 
we put $2\beta(x,\, y) = xy^{\iota} + yx^{\iota}$ for $x,\, y \in B$ 
with the main involution $\iota$ of $B$. 
Then the reduced norm $N_{B/F}(x)$ of $x$ is given by $\beta[x] = \beta(x,\, x)$. 
We denote by $D_{B}$ the discriminant of $B$ when $F$ is a number field. 

Every quaternion algebra over $F$ can be given by 
\begin{gather}
K + K\omega, \quad \omega^{2} = c, \quad x\omega = \omega x^{\iota} \quad \text{for every $x \in K$} 
\label{om}
\end{gather}
with a quadratic extension field $K$ of $F$ and $c \in F^{\times}$ (cf.\ \cite[{\S}1.10]{04}). 
We denote it by $\{K,\, c\}$. 
Then $\{K,\, c\} = M_{2}(F)$ if and only if $c \in \kappa[K^{\times}]$. 
In particular when $F$ is a nonarchimedean local field, 
a division quaternion algebra over $F$ is isomorphic to $\{K,\, c\}$ 
for an arbitrarily fixed element $c \in F^{\times}$ such that $c \not\in \kappa[K^{\times}]$. 
This is because there is a unique division quaternion algebra over $F$ up to isomorphisms; 
see \cite[Theorem 5.14]{04}, for example. 
We set 
\begin{gather}
B^{\circ} = \{x \in B \mid x^{\iota} = -x\},\quad 
	\beta^{\circ} = \beta |_{B^{\circ}}. \label{bc}
\end{gather}

\subsection{The invariants of $(W,\, \psi)$}

\begin{thm} \label{t1}
Let $(V,\, \p)$ be a quadratic space over an algebraic number field $F$ 
with invariants $\{n,\, F(\sqrt{\delta}),\, Q(\p),\, \{s_{v}(\p)\}_{v \in \mathbf{r}}\}$ 
and $n > 1$. 
Put $B = Q(\p)$. 
Given $q \in \p[V]\cap F^{\times}$, 
set $(V,\, \p) = (W,\, \psi) \oplus \langle F,\, q \rangle$. 
Then the invariants of $(W,\, \psi)$ are given by 
\begin{gather*}
\left\{n-1,\ F(\sqrt{(-1)^{n-1}\delta q}),\ Q(\psi),\ \{s_{v}(\psi)\}_{v \in \mathbf{r}}\right\}. 
\end{gather*} 
The characteristic quaternion algebra $Q(\psi_{v})$ at $v \in \mathbf{a} \cup \mathbf{h}$
and the index $s_{v}(\psi)$ at $v \in \mathbf{r}$ are determined as follows: 
\begin{enumerate}
\item Suppose $n$ is even. 
Put $K = F(\sqrt{\delta})$. 
\begin{enumerate}
\item If $n = 2$, 
then $Q(\psi) = M_{2}(F)$. 
\item If $n > 2$, 
then $Q(\psi_{v}) = M_{2}(F_{v})$ holds exactly in the following cases: 
\begin{eqnarray*}
	& & \xi_{v}(\delta) = 1 \text{ and } v \nmid D_{B}, \\
	& & \xi_{v}(\delta) \ne 1, v \nmid D_{B}, \text{ and } q \in \kappa[K_{v}^{\times}], \\
	& & \xi_{v}(\delta) \ne 1, v \mid D_{B}, \text{ and } q \not\in \kappa[K_{v}^{\times}], \\
	& & v \in \mathbf{r}, q_{v} > 0, \text{ and } s_{v}(\p) \equiv 0,\ 2 \pmod{8}, \\
	& & v \in \mathbf{r}, q_{v} < 0, \text{ and } s_{v}(\p) \equiv 0,\ 6 \pmod{8}, \\
	& & v \in \mathbf{a} \text{ such that } v \not\in \mathbf{r}. 
\end{eqnarray*}
\end{enumerate}
\item Suppose $n$ is odd. 
Put $K = F(\sqrt{\delta q})$. 
\begin{enumerate}
\item If $n = 3$, 
then $Q(\psi_{v}) = M_{2}(F_{v})$ holds exactly in the following cases:  
\begin{eqnarray*}
	& & \xi_{v}(\delta q) = 1 \text{ and } v \nmid D_{B}, \\
	& & \xi_{v}(\delta q) \ne 1, v \nmid D_{B}, \text{ and } \delta \in \kappa[K_{v}^{\times}], \\
	& & v \mid D_{B} \text{ and } \delta \not\in \kappa[K_{v}^{\times}], \\
	& & v \in \mathbf{r}, q_{v} > 0, \text{ and } s_{v}(\p) \equiv 1,\ 3 \pmod{8}, \\
	& & v \in \mathbf{r}, q_{v} < 0, \text{ and } s_{v}(\p) \equiv \pm 1 \pmod{8}, \\
	& & v \in \mathbf{a} \text{ such that } v \not\in \mathbf{r}. 
\end{eqnarray*}
\item If $n > 3$, 
then $Q(\psi_{v}) = M_{2}(F_{v})$ holds exactly in the following cases:  
\begin{eqnarray*}
	& & \xi_{v}(\delta q) = 1 \text{ and } v \nmid D_{B}, \\
	& & \xi_{v}(\delta q) \ne 1, v \nmid D_{B}, \text{ and } \delta \in \kappa[K_{v}^{\times}], \\
	& & \xi_{v}(\delta q) \ne 1, v \mid D_{B}, \text{ and } \delta \not\in \kappa[K_{v}^{\times}], \\
	& & v \in \mathbf{r}, q_{v} > 0, \text{ and } s_{v}(\p) \equiv 1,\ 3 \pmod{8}, \\
	& & v \in \mathbf{r}, q_{v} < 0, \text{ and } s_{v}(\p) \equiv \pm 1 \pmod{8}, \\
	& & v \in \mathbf{a} \text{ such that } v \not\in \mathbf{r}. 
\end{eqnarray*}
\end{enumerate}
\item For $v \in \mathbf{r}$, 
\begin{gather}
s_{v}(\psi) = 
	\begin{cases}
	s_{v}(\p) - 1 & \text{if $q_{v} > 0$}, \\ 
	s_{v}(\p) + 1 & \text{if $q_{v} < 0$}. 
	\end{cases} \label{ind}
\end{gather} 
\end{enumerate}
\end{thm}
In the proof of \THM{t1} we determine only two invariants $s_{v}(\psi)$ and $Q(\psi)$, 
since the other invariants can easily be seen. 
The determination of $Q(\psi)$ will be done in {\S}{\S}1.3 and 1.4. 
\\ 

We start with $s_{v}(\psi)$ for $v \in \mathbf{r}$. 
By definition we put $s_{v}(\p) = i_{v} - j_{v}$ if $\p_{v} = 1_{i_{v}} \oplus (-1_{j_{v}})$ 
with $0 \le i_{v},\, j_{v} \in \mathbf{Z}$. 
Since $\p_{v} = \psi_{v} \oplus q_{v}$ with $q_{v} \in \mathbf{R}^{\times}$, 
we have \REF{ind}. 
From this we can determine $Q(\psi)_{v}$ at $v \in \mathbf{r}$ by \REF{chre}, 
which depends only on $s_{v}(\psi)$. 
As for the archimedean prime $v \not\in \mathbf{r}$, 
we have $Q(\psi_{v}) = M_{2}(\mathbf{C})$. 
\\

Before observing $Q(\psi)$, 
we reduce the problem to that on a core subspace of $W_{v}$ at $v \in \mathbf{h}$ 
by retaking $h$ under $SO^{\p_{v}}$. 
We first recall that 
$Q(\psi)$ is determined by $Q(\psi_{v})$ for every $v \in \mathbf{a} \cup \mathbf{h}$. 
Since the archimedean case is given above, 
we consider only the nonarchimedean case. 
Hereafter until the end of Section 1, 
we fix a nonarchimedean prime $v$ and drop the subscript $v$ of local symbols below; 
let $F$ be a nonarchimedean local field, 
$\g$ its maximal order, 
and $\mathfrak{p}$ the prime ideal in $F$. 

Let $(V,\, \p)$ be a quadratic space with invariants $\{n,\, F(\sqrt{\delta}),\, B\}$, 
and put $W = (Fh)^{\perp}$ with $h \in V$ so that $\p[h] = q \in F^{\times}$. 
Notice that if $k$ is an element of $V$ such that $\p[k] = q$, 
then $(W,\, \psi)$ is isomorphic to $((Fk)^{\perp},\, \p\mid_{(Fk)^{\perp}})$ under some $\gamma \in SO^{\p}$ 
and both characteristic algebras are the same. 
Hence, 
to determine $Q(\psi)$, 
we may identify $h$ with a suitable element $k$ and $W$ with $(Fk)^{\perp}$. 
In the setting of \REF{w1} we shall do it according to whether or not $h$ can be taken from $Z$. 

If $U$ is an anisotropic subspace of $V$ containing $h$, 
then $\p$ is nondegenerate on $U$, 
$V = U \oplus U^{\perp}$, 
and $W = (W \cap U) \oplus U^{\perp}$, 
where $U^{\perp} = \{x \in V \mid \p(x,\, U) = 0\}$. 
Hence if there exists $h$ in $Z$ such that $\p[h] = q$, 
that is, 
$q \in \p[Z]$, 
we have a Witt decomposition 
\begin{gather}
W = (W \cap Z) \oplus \sum_{i=1}^{r} (Fe_{i} + Ff_{i}) \label{wc1}
\end{gather}
with a core subspace $W \cap Z$ of dimension $t - 1$. 

If $q \not\in \p[Z]$, 
taking $h = qe_{1} + f_{1}$ in $Fe_{1} + Ff_{1}$, 
we see that $f_{1} - qe_{1} \in W$ and 
$Z \oplus F(f_{1} - qe_{1})$ must be an anisotropic space of dimension $t + 1$. 
Then we have a Witt decomposition 
\begin{gather}
W = (Z \oplus F(f_{1} - qe_{1})) \oplus \sum_{i=2}^{r} (Fe_{i} + Ff_{i}). \label{wc2}
\end{gather}
Therefore, 
it is sufficient to investigate the core subspace of $W$ in \REF{wc1} or \REF{wc2} 
because the characteristic algebra of the core space gives $Q(\psi)$. 
In the next two subsections we write $s$ for the core dimension of $W$. 

\subsection{The even-dimensional case}

In this case, 
the dimension of $W$ is odd. 
Then we have only to observe the core dimension $s$ of $W$, 
because $Q(\psi)$ is a division algebra if and only if $s = 3$. 
Put $K = F(\sqrt{\delta})$. 

If $n = 2$, 
then $Q(\psi) = M_{2}(F)$ by definition. 

Suppose $n > 2$. 
If $t = 0$, 
then $K = F$ and $B = Q(\p) = M_{2}(F)$. 
Clearly $s = 1$, 
so that $Q(\psi) = M_{2}(F)$. 

If $t = 4$, 
then $K = F$, 
$B$ is a division quaternion algebra, 
and $(Z,\, \p)$ can be identified with $(B,\, \beta)$ (cf.\ \cite[Theorem 7.5]{04}). 
Since $\beta[B^{\times}] = F^{\times}$, 
we have $s = 3$ and hence $Q(\psi) = B$. 

If $t = 2$, 
then $K \ne F$ and $(Z,\, \p)$ can be identified with $(K,\, c\kappa)$ with some $c \in F^{\times}$; 
see {\S}1.4 below. 
If $q \in c\kappa[K^{\times}]$, 
then $s = 1$, 
so that $Q(\psi) = M_{2}(F)$; 
otherwise, 
$Q(\psi) \ne M_{2}(F)$. 
Moreover, 
$B = Q(\p)$ is given by $\{K,\, c\}$. 
Since $[F^{\times} : N_{K/F}(K^{\times})] = 2$ by local class field theory, 
$q \in c\kappa[K^{\times}]$ if and only if $B = M_{2}(F)$ and $q \in \kappa[K^{\times}]$, 
or if $B \ne M_{2}(F)$ and $q \not\in \kappa[K^{\times}]$. 
Combining these, 
we have the assertions in the case $K \ne F$. 

\subsection{The odd-dimensional case}

In this case, the dimension of $W$ is even. 
Put $K = F(\sqrt{\delta q})$. 
By \REF{cd1}, 
$K \ne F$ if and only if $s = 2$; 
$Q(\psi) = M_{2}(F)$ if $s = 0$ and $Q(\psi)$ is a division algebra if $s = 4$. 
Also if $s = 2$, 
a core space $(Y,\, \psi)$ of $W$ can be identified with $(K,\, c\kappa)$ with some $c \in F^{\times}$. 
Then $Q(\psi) = A(c\kappa) = \{K,\, c\}$. 
Such an element $c$ can be given as follows (cf. \cite[{\S}7.2]{04}): 
Take an orthogonal basis $\{x,\, y\}$ of $Y$ with respect to $\psi$. 
The map $ax + by \mapsto b + c^{-1}a\sqrt{-\psi[x]\psi[y]}$ 
gives an isomorphism of $(Y,\, \psi)$ onto $(K,\, c\kappa)$ with $c = \psi[y] \ne 0$. 
Then $A^{+}(Y) \cong K$ and $A(Y) \cong K + Ky = \{K,\, c\}$, 
which is the characteristic algebra $Q(\psi)$. 

We first consider the three-dimensional case. 
\begin{lem} \label{l1}
Suppose $n = 3$ and $F$ is a number field. 
For $q \in F^{\times}$, 
$q$ belongs to $\p[V]$ if and only if the following three conditions hold: 
\begin{enumerate}
\item $q_{v} > 0$ for $v \in \mathbf{r}$ such that $\p_{v}$ is positive definite, 
\item $q_{v} < 0$ for $v \in \mathbf{r}$ such that $\p_{v}$ is negative definite, 
\item $\xi_{v}(\delta q) \ne 1$ for $v \in \mathbf{h}$ such that $v \mid D_{B}$. 
\end{enumerate}
\end{lem}
\begin{proof}
This can be seen from the Hasse principle. 
In particular, 
nonarchimedean case can be verified by applying \cite[Lemma 4.2(1)]{99b}. 
\end{proof}

Suppose $n = 3$. 
If $t = 1$, 
then $B = M_{2}(F)$ and $(Z,\, \p) = \langle F,\, \delta\rangle $. 
Since $q \in \p[Z]$ if and only if $K = F$, 
we have 
\begin{gather*}
Q(\psi) = 
\begin{cases}
M_{2}(F) & \text{if $K = F$}, \\ 
\{K,\, \delta\} & \text{if $K \ne F$}. 
\end{cases}
\end{gather*} 
In fact, 
if $K \ne F$, 
then $q \not\in \p[Z]$, 
so that $Fg \oplus F(f_{1} - qe_{1})$ is a core space of $W$ with $g \in Z$ such that $\p[g] = \delta$. 
Thus $Q(\psi) = \{K,\, \delta\}$. 
From this we see that $Q(\psi) = M_{2}(F)$ if $\delta \in \kappa[K^{\times}]$ 
and $Q(\psi)$ is a division algebra if $\delta \not\in \kappa[K^{\times}]$. 

If $t = 3$, 
then $B$ is a division algebra and $(V,\, \p) = (Z,\, \p)$ can be identified with 
$(B^{\circ},\, -\delta \beta^{\circ})$ (cf.\ \cite[{\S}3.2]{cq}\ or\ \cite[{\S}7.3]{04}). 
Clearly, 
$q \in \p[Z]$ and so $s = 2$. 
We note that $q \in \p[Z]$ if and only if $K \ne F$ by \LEM{l1}(3). 
Since $W$ is an anisotropic space of dimension $2$ and $Q(\psi) = A(\psi)$, 
we need an explicit orthogonal basis of $W$. 
To find such a basis, 
we embed $K$ into $B$. 
Then $B$ can be written as $B = K + K\omega = \{K,\, c\}$ 
with a fixed element $c \in F^{\times}$ so that $c \not\in \kappa[K^{\times}]$. 
Put $k = \sqrt{\delta^{-1}q} \in K^{\times}$; 
then $B^{\circ} = Fk \oplus (F\omega + Fk\omega)$. 
Since $\p[k] = -\delta \beta^{\circ}[k] = q$, 
(changing a given $h$ to $k$, )
we have $W = F\omega + Fk\omega$. 
This basis satisfies $2\beta(\omega,\, k\omega) = -c\mathrm{Tr}_{K/F}(k^{\iota}) = 0$, 
and thus $\{\omega,\, k\omega \}$ is an orthogonal basis of $W$. 
Hence, 
in view of $c \not\in \kappa[K^{\times}]$, 
we see that 
\begin{gather}
Q(\psi) = \{K,\, \p[\omega]\} = \{K,\, \delta c\} = 
	\begin{cases}
	M_{2}(F) & \text{if $\delta \not\in \kappa[K^{\times}]$}, \\
	B & \text{if $\delta \in \kappa[K^{\times}]$}. 
	\end{cases} 
	\label{q}
\end{gather}

Next suppose $n > 3$. 
If $t = 1$, 
we can determine $Q(\psi)$ in the same way as in the case $n = 3,\, t = 1$. 

If $t = 3$, 
then $B$ is a division algebra. 
Assume that $q \in \p[Z]$; 
then $s = 2$ and $K \ne F$. 
Taking $B = \{K,\, c\}$ and $k \in K^{\times}$ with $c \not\in \kappa[K^{\times}]$ 
as considered in the case $n = t = 3$, 
we have a core subspace $F\omega \oplus Fk\omega$ of $W$ of dimension $2$. 
Then $Q(\psi)$ can be given by \REF{q}. 
Assume that $q \not\in \p[Z]$; 
then $s = 4$ and $K = F$. 
Hence we have a division algebra $Q(\psi) = B$. 

Summing up all these results, 
we obtain \THM{t1}. 

\section{Theorem on $[M/L \cap W]$ in the local case}

\subsection{Discriminant ideal}

Let $(V,\, \p)$ be a quadratic space over an algebraic number field or a nonarchimedean local field $F$. 
For a $\g$-lattice $L$ in $V$, 
we put 
\begin{gather*}
\widetilde{L} = L\,\widetilde{\,} = \{ x \in V \mid 2\p(x,\, L) \subset \g \}. 
\end{gather*}
We note that $L \subset \widetilde{L}$ if $\p[L] \subset \g$. 
By a \textit{$\g$-maximal} lattice $L$ with respect to $\p$, 
we understand a $\g$-lattice $L$ in $V$ 
which is maximal among $\g$-lattices on which the values $\p[x]$ are contained in $\g$. 

For $q \in F$ and a $\g$-ideal $\mathfrak{b}$ of $F$, 
we put 
\begin{gather*}
L[q] = \{ x \in L \mid \p[x] = q \}, \quad 
L[q,\, \mathfrak{b}] = \{ x \in V \mid \p[x] = q,\ \p(x,\, L) = \mathfrak{b} \}. 
\end{gather*}

For two $\g$-lattices $L$ and $M$ in $V$, 
we denote by $[L/M]$ a $\g$-ideal of $F$ generated over $\g$ by 
$\det(\alpha)$ of all $F$-linear automorphisms $\alpha$ of $V$ 
such that $L\alpha \subset M$. 
If $F$ is a global field, 
then $[L/M] = \prod_{v \in \mathbf{h}} [L_{v}/M_{v}]$ 
with the localization $[L/M]_{v} = [L_{v}/M_{v}]$ at each $v$. 
We call $[\widetilde{L}/L]$ the \textit{discriminant ideal} of $(V, \p)$ 
if $L$ is a $\g$-maximal lattice in $V$ with respect to $\p$. 
This is an invariant of $(V,\, \p)$ and independent of the choice of $L$, 
which follows from the basic fact that 
all $\g$-maximal lattices form a single class with respect to $SO^{\p}(V)$ when $F$ is a local field. 
If $F$ is a local field, 
by \LEM{lid}(4) below, 
the discriminant ideal of $\p$ coincides with that of a core space of $\p$. 

In the later argument we will often need the discriminant ideals of local spaces, 
which can be obtained by applying \cite[Theorem 6.2(ii)\ and\ (iii)]{cq} with suitable localization. 
We restate that theorem due to Shimura: 

\begin{thm} $\mathrm{(Shimura\, \cite[Theorem\ 6.2(ii)\ and\ (iii)]{cq})}$ \label{di}
Given $(V,\, \p)$ be as above with an algebraic number field $F$, 
suppose that $L$ is a $\g$-maximal lattice in $V$ with respect to $\p$. 
Let $K=F(\sqrt{\delta})$ be the discriminant field of $\p$ and 
$\mathfrak{e}$ the product of the prime ideals of $F$ ramified in $Q(\p)$. 
\begin{enumerate}
\item Suppose $n$ is even. 
Let $\mathfrak{e}_{1}$ be the product of 
all prime factors of $\mathfrak{e}$ that do not ramify in $K$. 
Then $[\widetilde{L}/L] = D_{K/F}\mathfrak{e}_{1}^{2}$. 
Here we understand that $D_{K/F} = \g$ if $K = F$. 
\item Suppose $n$ is odd. 
Put $\delta \g = \mathfrak{a}\mathfrak{b}^{2}$ 
with a squarefree integral ideal $\mathfrak{a}$ and a fractional ideal $\mathfrak{b}$. 
Then $[\widetilde{L}/L] = 2\mathfrak{a}^{-1}\mathfrak{e}^{2} \cap 2\mathfrak{a}$. 
\end{enumerate}
\end{thm}

\begin{lem} \label{lid}
For $\g$-lattices $L$ and $M$ in $V$, 
except (5), 
the following assertions hold: 
\begin{enumerate}
\item There exists $\alpha \in GL(V)$ such that $M\alpha = L$. 
Moreover, if $\alpha$ is such an element, 
then $[M/L] = \det(\alpha)\g$. 
\item If $N$ is a $\g$-lattice in $V$, 
then $[N/L] = [N/M][M/L]$. 
\item If $L \subset M$ and $\p[M] \subset \g$, 
then $[\widetilde{L}/L] = [M/L]^{2}[\widetilde{M}/M]$ and $[\widetilde{L}/\widetilde{M}] = [M/L]$. 
\item If $L \subset M$, 
then $[M/L] \subset \g$. 
Moreover, 
$L = M$ if and only if $[M/L] = \g$. 
\item Suppose $V = U \oplus U^{\perp}$ with a subspace $U$ of $V$. 
Let $L$ be a $\g$-lattice in $U$ and $M$ a $\g$-lattice in $U^{\perp}$. 
Then $[(L + M)\widetilde{\,}/L + M] = [\widetilde{L}/L][\widetilde{M}/M]$. 
\item Suppose that $\p[L] \subset \g$ and $M$ is $\g$-maximal with respect to $\p$. 
Then $[M/L]$ is independent of the choice of $M$. 
Moreover, 
$[M/L] \subset \g$ and $[\widetilde{L}/L] = [M/L]^{2}[\widetilde{M}/M]$. 
\item Under the assumptions in (6), 
$L$ is $\g$-maximal if and only if $[M/L] = \g$, 
or equivalently, 
if $[\widetilde{L}/L] = [\widetilde{M}/M]$. 
\end{enumerate}
\end{lem}
\begin{proof}
Each assertion in the global case can be reduced to that in the local case by localization. 
Here we note that $L$ is maximal if and only if $L_{v}$ is maximal for every $v \in \mathbf{h}$. 
Hence we prove the assertions, 
assuming $F$ is a nonarchimedean local field. 

To see (1), 
by considering $\g$-bases of $L$ and $M$, 
we can take $0 \ne c \in \g$ so that $cL \subset M$. 
Then there exists a $\g$-basis $\{k_{i}\}_{i=1}^{n}$ of $M$ such that 
$cL = \g \varepsilon_{1}k_{1} + \cdots + \g \varepsilon_{n}k_{n}$ 
with the elementary divisors $\varepsilon_{1}\g,\, \cdots,\, \varepsilon_{n}\g$, 
where $n$ is the dimension of $V$. 
Employing these expressions of $cL$ and $M$, 
we can find an $F$-linear automorphism $\alpha$ of $V$ such that $M\alpha = cL$. 
Then $c^{-1}\alpha$ is the required surjection. 
As for the second assertion of (1), 
clearly $\det(\alpha)\g \subset [M/L]$. 
Let $\gamma \in GL(V)$ such that $M\gamma \subset L$. 
Then $M\gamma \alpha^{-1} \subset M$, 
which implies $\det(\gamma \alpha^{-1}) \in \g$. 
Thus we have $\det(\gamma) \in \det(\alpha)\g$, 
so that $[M/L] \subset \det(\alpha)\g$ as desired. 
Assertion (2) can be seen from (1). 

To see (3), 
let $L \subset M$ and $\p[M] \subset \g$; 
then $L \subset M \subset \widetilde{M} \subset \widetilde{L}$. 
Since $L \subset M$, 
we can find a $\g$-basis $\{k_{i}\}_{i=1}^{n}$ of $M$ such that 
$L = \g \varepsilon_{1}k_{1} + \cdots + \g \varepsilon_{n}k_{n}$ 
with some $\{\varepsilon_{i}\}_{i=1}^{n}$. 
Then $(V,\, \p)$ is isomorphic to $(F^{1}_{n},\, \p_{0})$ 
with the matrix $\p_{0}$ representing $\p$ with respect to $\{k_{i}\}_{i=1}^{n}$. 
Identifying $V$ with $F^{1}_{n}$ under this isomorphism, 
we see that $M = \g^{1}_{n}$, 
$L = \g^{1}_{n}\varepsilon$, 
$\widetilde{M} = \g^{1}_{n}(2\p_{0})^{-1}$, 
and $\widetilde{L} = \g^{1}_{n}(2\p_{0}\varepsilon)^{-1}$, 
where $\varepsilon = \mathrm{diag}[\varepsilon_{1},\, \cdots,\, \varepsilon_{n}]$. 
Since $[M/L] = \det(\varepsilon)\g$ and $[\widetilde{M}/M] = \det(2\p_{0})\g$, 
we have $[\widetilde{L}/\widetilde{M}] = [M/L]$ and 
$[\widetilde{L}/L] = [\g^{1}_{n}/\g^{1}_{n}\varepsilon 2\p_{0}\varepsilon] = [M/L]^{2}[\widetilde{M}/M]$. 
To see (4), 
with the notation in the above proof of (3), 
observe that $[M/L] \subset \g$ and $[M/L] = \g$ if and only if $\det(\varepsilon) \in \g^{\times}$. 
This proves (4). 

As for (5), 
by fixing $\g$-bases of $L$ and $M$, 
we can put $L = \g^{1}_{\ell}$, $M = \g^{1}_{m}$, and $\p = \mathrm{diag}[\p_{1},\, \p_{2}]$, 
where $\p_{1}$ (resp. $\p_{2}$) is the matrix representing the restriction of $\p$ 
to $U$ (resp. $U^{\perp}$) with respect to that basis. 
Then we have 
$[(L + M)\widetilde{\,}/L + M] = \det(2\mathrm{diag}[\p_{1},\, \p_{2}])\g = [\widetilde{L}/L][\widetilde{M}/M]$. 

To see (6), 
let $N$ be a $\g$-maximal lattice in $V$ with respect to $\p$. 
Since two maximal lattices are isomorphic under $SO^{\p}(V)$, 
we have $N = M\gamma$ with $\gamma \in SO^{\p}(V)$; 
see \cite[Lemma 6.9]{04}, for example. 
Then $[N/L] = [M/L]$, 
which implies that $[M/L]$ is independent of the choice of $M$. 
Moreover, 
if $\p[L] \subset \g$, 
we can take $N$ so that $L \subset N$, 
and hence by (3) we have 
$[\widetilde{L}/L] = [N/L]^{2}[\widetilde{N}/N] = [M/L]^{2}[\widetilde{M}/M]$ 
because the discriminant ideal does not depend on $N$. 
As for (7), 
for a $\g$-maximal lattice $N$ containing $L$, 
using (4) and (6), 
we have $L = N$ if $[M/L] = \g$. 
Conversely if $L$ is maximal, 
then $[M/L] = [L/L] = \g$. 
The last assertion follows from (6). 
\end{proof}

\subsection{The ideal $[M/L \cap W]$ in the local case}

Let $F$ be a nonarchimedean local field, 
$\g$ its maximal order, 
and $\mathfrak{p}$ the prime ideal of $F$. 
Let $(V,\, \p)$ be a quadratic space over $F$ with invariants $\{n,\, F(\sqrt{\delta}),\, B\}$ 
and $L$ a $\g$-maximal lattice in $V$ with respect to $\p$. 
Then there exists a Witt decomposition of $V$ as follows (cf.\ \cite[Lemma 6.5]{04}): 
\begin{gather}
V = Z \oplus \sum_{i=1}^{r} (Fe_{i} + Ff_{i}),\quad 
L = N + \sum_{i=1}^{r} (\g e_{i} + \g f_{i}), \label{w2} \\
N = \{x \in Z \mid \p[x] \in \g \},\quad 
\p(e_{i},\, e_{j}) = \p(f_{i},\, f_{j}) = 0,\quad 
2\p(e_{i},\, f_{j}) = \delta_{ij}. \nonumber 
\end{gather}
Here $Z$ is a core subspace of $(V,\, \p)$ and 
$N$ is a unique $\g$-maximal lattice in $Z$ with respect to $\p$. 
Let $t$ be the core dimension of $V$; 
then $n = t + 2r$ and $0 \le t \le 4$. 
\\

Hereafter we assume $n > 1$. 
For $h \in V$ such that $\p[h] = q \ne 0$, 
put $W = (Fh)^{\perp}$ defined by \REF{orc} and let $\psi$ be the restriction of $\p$ to $W$. 
Then the invariants of $(W,\, \psi)$ is given by 
$\{n-1,\, F(\sqrt{(-1)^{n-1}\delta q}),\, Q(\psi)\}$ and $Q(\psi)$ by \THM{t1}. 
\begin{thm} \label{t2}
Let $(V,\, \p)$ be a quadratic space of dimension $n\,(> 1)$ over a nonarchimedean local field $F$ 
and $L$ a $\g$-maximal lattice in $V$ with respect to $\p$. 
For an element $h$ of $L$ such that $\p[h] \ne 0$, 
put $q = \p[h]$ and $W = (Fh)^{\perp}$ and let $\psi$ be the restriction of $\p$ to $W$. 
Let $M$ be a $\g$-maximal lattice in $W$ with respect to $\psi$. 
Then there exists a nonnegative integer $\lambda(q)$ such that 
\begin{gather}
[M/L \cap W] = \mathfrak{p}^{\lambda(q)-m} \label{id1}
\end{gather}
if $2\p(h,\, L) = \mathfrak{p}^{m}$. 
\end{thm}
We note by \LEM{lid}(6) that $[M/L \cap W]$ is independent of the choice of $M$. 
\\

The proof will be done in this and next sections. 
We begin by explaining that $[M/L \cap W]$ can be dealt with in a specified setting. 

First of all, 
by \LEM{lid}(6) we have 
\begin{gather}
[(L \cap W)\widetilde{\,}/L \cap W] = [M/L \cap W]^{2}[\widetilde{M}/M]. \label{id2}
\end{gather}
Since $[\widetilde{M}/M]$ can be given by \THM{di}, 
our task is to determine either $[M/L \cap W]$ or $[(L \cap W)\widetilde{\,}/L \cap W]$. 

Let $h \in L$, 
$0 \ne q = \p[h] \in \g$, 
and take a Witt decomposition for $(V,\, \p)$ as in \REF{w2}. 
By virtue of Yoshinaga \cite[Theorem 3.5]{Y}, 
$L[q]$ is decomposed into finitely many sets $L[q,\, 2^{-1}\mathfrak{p}^{m}]$ as follows: 
\begin{eqnarray}
L[q] &=& 
	\begin{cases}
	\bigsqcup_{m = 0}^{\tau(q)} L[q,\, 2^{-1}\mathfrak{p}^{m}] & \text{if $r \ne 0$}, \\ 
	L[q,\, 2^{-1}\mathfrak{p}^{\tau(q)}] & \text{if $r = 0$}, 
	\end{cases} \label{re0} \\
L[q,\, 2^{-1}\mathfrak{p}^{m}] &=& k_{m}C(L), \quad \text{except when $t = 0$ and $r = 1$} \label{re} 
\end{eqnarray}
with explicitly given integer $\tau(q)\geq0$ and elements $k_{m}$ for $0 \le m \le \tau(q)$. 
Here 
\begin{gather*}
C(L) = \{ \gamma \in SO^{\p}(V) \mid L \gamma = L \}. 
\end{gather*}
Each $k_{m}$ is given as follows: 
\begin{enumerate}
\item If $t = 0$ and $r > 1$, 
then $k_{m} = q\pi^{-m}e_{1} + \pi^{m}f_{1} \in \g e_{1} + \g f_{1}$. 
We also take such a $k_{m}$ in the case $t = 0$ and $r = 1$; 
see Case(1) below. 
\item If $t \ne 0$ and $r = 0$, 
then $k_{m}$ is an arbitrarily fixed element $z$ of $L[q]$. 
\item If $t \ne 0$ and $r \ne 0$, 
then there are the following three cases; 
\begin{enumerate}
\item $k_{m} = q\pi^{-m}e_{1} + \pi^{m}f_{1} \in \g e_{1} + \g f_{1}$. 
\item $k_{m} = z + \pi^{-m}q(1-s)e_{1} + \pi^{m}f_{1} \in N + \g e_{1} + \g f_{1}$ 
for an arbitrarily fixed $z \in N[sq]$ with a suitable $s \in \g^{\times}$. 
\item $k_{m} = z + \pi^{m}e_{1} \in N + \g e_{1}$ for an arbitrarily fixed $z \in N[q]$. 
\end{enumerate}
\end{enumerate}
For the detailed conditions for the choice of $k_{m}$ in (3), 
and also those for $\tau(q)$, 
we will explain in Case (1), Case (2), and Section 3 below. 

If $2\p(h,\, L) = \mathfrak{p}^{m}$, 
then $h\gamma = k_{m}$ with some $\gamma \in C(L)$ by \REF{re}, 
so that $(L \cap W)\gamma = L \cap W_{0}$, 
where $W_{0} = (Fk_{m})^{\perp}$. 
Hence $[M/L \cap W] = [M\gamma /L \cap W_{0}]$. 
Since the latter ideal does not depend on the $M\gamma$, 
we have $[M\gamma /L \cap W_{0}] = [M_{0}/L \cap W_{0}]$ 
for an arbitrary $\g$-maximal lattice $M_{0}$ in $W_{0}$. 
Thus, 
by using an explicit element $k_{m}$, 
we can consider $[M_{0}/L \cap W_{0}]$ instead of $[M/L \cap W]$ for a given $h$. 

Set $U = Z \oplus (Fe_{1} + Ff_{1})$ and $W_{1} = W_{0} \cap U$. 
Since $k_{m}$ belongs to $U$ in any case above, 
we see that 
\begin{gather*}
W_{0} = W_{1} \oplus \sum_{i=2}^{r} (Fe_{i} + Ff_{i}), \quad 
L \cap W_{0} = (L \cap W_{1}) + \sum_{i=2}^{r} (\g e_{i} + \g f_{i}). 
\end{gather*}
Then by \LEM{lid}(5) we have 
$[(L \cap W_{0})\widetilde{\,}/L \cap W_{0}] = [(L \cap W_{1})\widetilde{\,}/L \cap W_{1}]$. 
Therefore, 
in view of this fact, 
it is sufficient to prove our theorem in the setting \REF{w2} 
with $U$ of dimension $t + 2$ in place of $V$ and with $W_{1}$ of dimension $t + 1$ in place of $W$. 

Hereafter, 
we put $h = k_{m}$ and $W = W_{1}$ for simplicity; 
also put $e = e_{1}$, $f = f_{1}$, 
and $q \in \pi^{\nu} \g^{\times}$ with $0 \le \nu \in \mathbf{Z}$. 
\\

\textbf{Case(1)}: $t = 0$ and $r > 1$. 
In this case, 
$V = Fe + Ff$, $L = \g e + \g f$, 
and $h = q\pi^{-m}e + \pi^{m}f \in L$ with $\tau(q) = [\nu/2]$. 
Since $\p(g,\, h) = 0$ for $g = \pi^{m}f - q\pi^{-m}e$, 
we see that $W = Fg$ and $M = \g \pi^{-[\nu/2]}g$ is a $\g$-maximal lattice in $W$. 
For $xg \in W$ with $x \in F$, 
$xg \in L$ if and only if $x \in \mathfrak{p}^{-m}$; 
hence $L \cap W = \g \pi^{-m}g$. 
Then we have $[M/L \cap W] = \pi^{[\nu/2]-m}\g = \mathfrak{p}^{[\nu/2]-m}$. 

As for the case $t = 0$ and $r = 1$, 
\cite[Theorem 3.2(2)]{Y} shows that 
$L[q,\, 2^{-1}\mathfrak{p}^{m}] = h_{1}C(L) \sqcup h_{2}C(L)$ if $m \ne \nu/2$, 
where $h_{1} = q\pi^{-m}e + \pi^{m}f$ and $h_{2} = q\pi^{m-\nu}e + \pi^{\nu-m}f$. 
So that we have to verify our theorem to the case $h = h_{2}$; 
but it can be seen in a similar way. 
\\

\textbf{Case(2)}: $t \ne 0$ and $r = 0$. 
In this case, 
$V = Z$ is anisotropic, 
$L = N$, 
and $2\p(h,\, L) = \mathfrak{p}^{\tau(q)}$ for $h \in L[q]$. 
Then $W$ is anisotropic. 
Let $M$ be a $\g$-maximal lattice in $W$, 
which is given by $M = \{x \in W \mid \p[x] \in \g \}$. 
Then we have $L \cap W = N \cap W = M$, 
and hence $[M/L \cap W] = \g$. 
This gives the required fact. 
Clearly $\lambda(q)$ of \REF{id1} is $\tau(q)$ in the present case, 
which will be stated in \REF{corid} after the proof of \THM{t2}. 
We also note in the present case that 
$C(L) = SO^{\p}$ and $C(L \cap W) = SO^{\psi}$ 
by the uniqueness of maximal lattice in the anisotropic space. 
\\

\textbf{Case(3)}: $t \ne 0$ and $r \ne 0$. 
In this case, 
we fix the core dimension $t$ and consider $[M/L \cap W]$ 
with the representative $h$ in (i), (ii), or (iii) of (3). 
Before starting the proof, 
we state an auxiliary lemma. 
\begin{lem} \label{l2}
Let the notation be as in \THM{t2}. 
In the setting \REF{w2} with $n = t + 2$ as explained above, 
the following assertions hold: 
\begin{enumerate}
\item Suppose $L \cap W \subset M$. 
Let $\{\varepsilon_{1}\g,\, \cdots,\, \varepsilon_{t+1}\g \}$ be the elementary divisors 
determined by $M$ and $L \cap W$. 
Then $[M/L \cap W] = \varepsilon_{1} \cdots \varepsilon_{t+1}\g$. 
\item Let $h = q\pi^{-m}e + \pi^{m}f \in \g e + \g f$ 
be as in (i) of Case (3) with $0 \le m \le [\nu/2]$. 
If $q \not\in \p[Z]$, 
then $[M/L \cap W] = \mathfrak{p}^{[\nu/2]-m}$. 
\item Let $h = z + \pi^{m}e \in N + \g e$ be as in (iii) of Case (3). 
Then 
\begin{gather}
W = (Z \cap (Fz)^{\perp}) \oplus (Fe + Ff_{0}), \quad 
f_{0} = f - \frac{\pi^{2m}}{4q}e - \frac{\pi^{m}}{2q}z, \label{f0}
\end{gather}
is a Witt decomposition of $(W,\, \psi)$. 
\item With the notation in assertion (3) let $N_{0}$ be the $\g$-maximal lattice in $Z \cap (Fz)^{\perp}$ 
and set $M = N_{0} + \g e + \g f_{0}$, 
which is a $\g$-maximal lattice in $W$. 
Take a $\g$-basis $\{k_{1},\, \cdots ,\, k_{t}\}$ of $N$ 
so that $N_{0} = \sum_{i=1}^{t-1} \g \varepsilon_{i}k_{i}$ 
with the elementary divisors $\{\varepsilon_{i}\g \}_{i=1}^{t-1}$. 
Put $z = \sum_{i=1}^{t} a_{i}k_{i} \in N$ and $\mathfrak{p}^{\lambda} = \g \cap 2q\pi^{-m}a_{t}^{-1}\g$. 
Then $[M/L \cap W] = (\varepsilon_{1} \cdots \varepsilon_{t-1})^{-1}\mathfrak{p}^{\lambda}$. 
In particular if $t = 1$, 
let $N = \g k_{1}$ and put $z = a_{1}k_{1} \in N$. 
Then $[M/L \cap W] = \g \cap 2q\pi^{-m}a_{1}^{-1}\g$.  
\end{enumerate}
\end{lem}
\begin{proof}
(1) By \LEM{lid}(1) we have 
$[M/L \cap W] = \det(\mathrm{diag}[\varepsilon_{1} \cdots \varepsilon_{t+1}])\g$ 
with a suitable $\g$-basis of $M$. 

(2) Taking $g = \pi^{m}f - q\pi^{-m}e$, 
we have $W = Z \oplus Fg$, 
which is anisotropic because $q \not\in \p[Z]$. 
Then $M = N + \g\pi^{-[\nu/2]}g$ is a $\g$-maximal lattice in $W$ with respect to $\psi$. 
For $x = z + x_{1}g \in W$ with $z \in Z$ and $x_{1} \in F$, 
we see that $x \in L$ if and only if $z \in N$ and $x_{1} \in \mathfrak{p}^{-m}$. 
Hence $L \cap W = N + \g \pi^{-m}g$. 
Then we have $[M/L \cap W] = \pi^{[\nu/2]-m}\g = \mathfrak{p}^{[\nu/2]-m}$. 

(3) Put $U = Z \cap (Fz)^{\perp}$ with a fixed $z \in N[q]$. 
Clearly this is an anisotropic space of dimension $t-1$. 
Since $\p(e,\, h) = \p(f-(2q)^{-1}\pi^{m}z,\, h) = 0$, 
we have $W = U \oplus (Fe + F(f-(2q)^{-1}\pi^{m}z))$. 
Then $f_{0} = f - (4q)^{-1}\pi^{2m}e - (2q)^{-1}\pi^{m}z$ belongs to $U^{\perp} \cap W$ 
and satisfies $\p[f_{0}] = 0$ and $\p(e,\, f_{0}) = 2^{-1}$. 
Thus $U^{\perp} \cap W = Fe + Ff_{0}$, 
which gives a Witt decomposition of $W$. 

(4) We first note that $N_{0} = N \cap (Fz)^{\perp}$ 
because $U$ in the proof of (3) is anisotropic. 
Also, 
$a_{t} \ne 0$ as $z \not\in U$. 
For $x = \sum_{i=1}^{t-1} x_{i}\varepsilon_{i}k_{i} + x_{t}e + x_{t+1}f_{0} \in W$ with $x_{i} \in F$, 
we see that 
\begin{eqnarray*}
x \in L &\Longleftrightarrow& 
	\begin{cases}
	\sum_{i=1}^{t-1} (x_{i}\varepsilon_{i} - (2q)^{-1}\pi^{m}x_{t+1}a_{i})k_{i} 
		- (2q)^{-1}\pi^{m}x_{t+1}a_{t}k_{t} \in N, & \\
	(x_{t} - (4q)^{-1}\pi^{2m}x_{t+1})e + x_{t+1}f \in \g e + \g f & \\
	\end{cases} \\ 
&\Longleftrightarrow& 
	[x_{1},\, \cdots ,\, x_{t+1}] \in \g^{1}_{t+1}\alpha_{0}, 
\end{eqnarray*}
where $\alpha_{0}$ is a triangular matrix in $GL_{t+1}(F)$ given by 
\begin{gather*}
\begin{bmatrix}
   		\varepsilon_{1}^{-1} &  &  &  &  \\ 
   		0 & \ddots &  &  &  \\ 
   		\vdots & \ddots & \varepsilon_{t-1}^{-1} &  &  \\ 
   		0 & \cdots & 0 & 1 &  \\ 
   		\pi^{m+\lambda}a_{1}(2q\varepsilon_{1})^{-1} & \cdots & 
			\pi^{m+\lambda}a_{t-1}(2q\varepsilon_{t-1})^{-1} & 
			\pi^{2m+\lambda}(4q)^{-1} & \pi^{\lambda} \\ 
\end{bmatrix}.
\end{gather*} 
Thus the automorphism $\alpha$ of $W$, 
represented by $\alpha_{0}$ with respect to the basis $\{\varepsilon_{i}k_{i},\, e,\, f_{0}\}$, 
gives a surjection of $M$ onto $L \cap W$. 
Therefore by \LEM{lid}(1) we find that 
$[M/L \cap W] = \det(\alpha_{0})\g = (\varepsilon_{1} \cdots \varepsilon_{t-1})^{-1}\mathfrak{p}^{\lambda}$. 
Similarly for $t = 1$, 
we can find an automorphism of $W$ represented by 
$\begin{bmatrix}
   		1 & 0 \\ 
		\pi^{2m+\lambda}(4q)^{-1} & \pi^{\lambda} \\
\end{bmatrix}
$ 
with respect to the basis $\{e,\, f_{0}\}$, 
and it gives the desired ideal $[M/L \cap W]$. 
\end{proof}

In the next section we denote by $F(\sqrt{\delta})$ the discriminant field of $\p$; 
we may assume that $\delta \in \g^{\times} \cup \pi\g^{\times}$ by a suitable base change of $V$. 
Also we will consider the discriminant field $K$ of $\p$ or $\psi$ 
according to the core dimension $t$. 
Hereafter, 
we put $D_{K/F} = \pe^{d}$ if $K$ is a ramified extension of $F$, 
and put $d = 1$ if $K$ is an unramified quadratic extension of $F$ for convenience. 
Put also $2\g = \pe^{\kappa}$, 
which may not be confused with $\kappa$ as the norm form of $K$. 

Let us recall some facts on quadratic fields of $F$, 
which will be often used in our argument: 
Assume $2 \in \pe$ and let $b \in \g^{\times}$. 
Then by applying \cite[Lemma\ 3.5(2)\ and\ (3)]{99b}, 
we see that 
\begin{eqnarray}
b &\in& (1 + \pi^{2\kappa}\g^{\times})\g^{\times 2} \quad \text{if $\xi(b) = -1$}, \label{35u} \\
b &\in& (1 + \pi^{2k+1}\g^{\times})\g^{\times 2}\quad \text{and} \nonumber \\
b &\not\in& (1 + \pe^{\varepsilon})\g^{\times 2}\quad \text{for every $\varepsilon > 2k+1$ if $\xi(b) = 0$} 
\label{35r}
\end{eqnarray}
with some $0 \le k < \kappa$. 
Moreover by \cite[Lemma\ 3.5(5)]{99b}, 
$D_{F(\sqrt{b})/F} = \pe^{2(\kappa-k)}$ if $\xi(b) = 0$. 

When $K/F$ is a ramified quadratic extension and $2 \in \pe$, 
from local class field theory we also recall that 
\begin{gather}
[F^{\times}:N_{K/F}(K^{\times})] = [\g^{\times}:N_{K/F}(\mr^{\times})] = 2, \nonumber \\ 
\g^{\times} = (1 + \pe^{d-1})N_{K/F}(\mr^{\times}), \label{cd}
\end{gather}
where $\mr$ is the maximal order in $K$. 

\section{Proof of \THM{t2} in Case(3)}

\subsection{Case $t = 1$}

Put $K = F(\sqrt{\delta q})$ and $\xi = \xi(\delta q)$. 
We first observe that $q \in \p[Z]$ if and only if $K = F$. 
By \cite[Theorem 3.5(ii)]{Y}, 
we can take $h$ and $\tau(q)$ in Case(3) as follows: 
\begin{gather*}
h = 
\begin{cases}
z + \pi^{m}e & \text{if $q \in \p[Z]$}, \\ 
z + \pi^{-m}q(1-s)e + \pi^{m}f & 
	\text{if $q \not\in \p[Z]$, $2 \in \pe$, and $\ord(\delta q)$ is even}, \\ 
q\pi^{-m}e + \pi^{m}f & \text{otherwise}, 
\end{cases} \\ 
\tau(q) = 
\begin{cases}
\kappa + (\nu+\ord(\delta))/2 & \text{if $q \in \p[Z]$}, \\ 
\kappa + [(\nu+1-d)/2] & 
	\text{if $q \not\in \p[Z]$, $2 \in \pe$, and $\ord(\delta q)$ is even}, \\ 
[\nu/2] & \text{otherwise}. 
\end{cases}
\end{gather*}

Suppose that $q \in \p[Z]$. 
Then $K = F$ and $h = z + \pi^{m}e \in N + \g e$ 
with $0 \le m \le \kappa + (\nu+\ord(\delta))/2$ and a fixed element $z$ of $N[q]$. 
By \LEM{l2}(3), 
$W = Fe + Ff_{0}$ is a Witt decomposition, 
and so $M = \g e + \g f_{0}$ is $\g$-maximal in $W$. 
Since $L = \g\pi^{-[\nu/2]}z + \g e + \g f$, 
applying \LEM{l2}(4) to $k_{1} = \pi^{-[\nu/2]}z$, 
we have $[M/L \cap W] = 2q\pi^{-m-[\nu/2]}\g = \pe^{(\nu+\ord(\delta))/2+\kappa-m}$. 
\\

Suppose that $q \not\in \p[Z]$. 
Then $K \ne F$ and $W$ is anisotropic of dimension $2$. 
If $2 \not\in \mathfrak{p}$ or if $\ord(\delta q)$ is odd, 
then $h = q\pi^{-m}e + \pi^{m}f \in \g e + \g f$ with $0 \le m \le [\nu/2]$. 
Since $q \not\in \p[Z]$, 
we have $[M/L \cap W] = \mathfrak{p}^{[\nu/2]-m}$ by \LEM{l2}(2). 

If $2 \in \mathfrak{p}$ and $\ord(\delta q)$ is even, 
take $g \in Z$ so that $\p[g] = \delta$. 
Since $\delta q\pi^{-\ord(\delta q)} \in \g^{\times}$ and $K \ne F$, 
by applying \REF{35u} or \REF{35r} according as whether or not $K$ is unramified over $F$, 
we may put $\delta q\pi^{-\ord(\delta q)} = sb^{2}$ 
with $s \in 1 + \pi^{2\kappa+1-d}\g^{\times}$ and $b \in \g^{\times}$. 
Then $K = F(\sqrt{s})$ and $z = \delta^{-1}bs\pi^{[(\nu+1)/2]}g \in \g g$, 
which satisfies $\p[z] = sq$. 
For such $s \in \g^{\times}$ and $z \in N[sq]$, 
we can take $h = z + \pi^{-m}q(1-s)e + \pi^{m}f \in N + \g e + \g f$ 
with $0 \le m \le \kappa + [(\nu+1-d)/2]$. 
Since $Z = Fz$ and $\p[z] = sq \in \pi^{\nu}\g^{\times}$, 
we have $L = \g\pi^{-[\nu/2]}z + \g e + \g f$. 
Put $g_{1} = \pi^{-2m}q(s-1)e + f$ and $g_{2} = e - (2sq)^{-1}\pi^{m}z$; 
then $W = Fg_{1} + Fg_{2}$. 
For $x = x_{1}g_{1} + x_{2}g_{2} \in W$ with $x_{i} \in F$, 
we see that $x \in L$ if and only if $x_{1} \in \g$ and $x_{2} \in \pe^{[(\nu+1)/2]+\kappa-m}$, 
and hence $L \cap W = \g g_{1} + \g \pi^{[(\nu+1)/2]+\kappa-m}g_{2}$. 
Observing the matrix representing $\p$ with respect to this basis, 
we find that $[(L \cap W)\widetilde{\,}/L \cap W] = \pe^{2([(\nu+1)/2]+\kappa-m)}$. 
On the other hand, 
since $Q(\p) = M_{2}(F)$, 
the invariants of $W$ are $\{2,\, K,\, \{K,\, \delta\}\}$ by \THM{t1}. 
Thus \THM{di} gives 
\begin{gather*}
[\widetilde{M}/M] = 
\begin{cases}
\g &  \text{if $\xi = -1$ and $\delta \in \kappa[K^{\times}]$}, \\
\pe^{2} &  \text{if $\xi = -1$ and $\delta \not\in \kappa[K^{\times}]$}, \\
\pe^{d} &  \text{if $\xi = 0$}. 
\end{cases}
\end{gather*}
Combining these with \REF{id2}, 
we have 
\begin{gather*}
[M/L \cap W] = 
\begin{cases}
\pe^{[(\nu+1)/2]+\kappa-m} &  \text{if $\xi = -1$, $\delta \in \g^{\times}$, and $\nu$ is even}, \\
\pe^{[(\nu+1)/2]+\kappa-1-m} &  \text{if $\xi = -1$, $\delta \in \pi\g^{\times}$, and $\nu$ is odd}, \\
\pe^{[(\nu+1)/2]+\kappa-d/2-m} &  \text{if $\xi = 0$}. 
\end{cases}
\end{gather*}

\subsection{Case $t = 2$}

Put $K = F(\sqrt{\delta})$ and $\xi = \xi(\delta)$. 
Then $K \ne F$ and $(Z,\, \p)$ can be identified with $(K,\, c\kappa)$ with some $c \in F^{\times}$. 
The isomorphism class of $(Z,\, \p)$ is determined by $K$ and $cN_{K/F}(K^{\times})$. 
Hence we can take such a $c$ from $\g^{\times} \cup \pi\g^{\times}$ if $K$ is unramified over $F$, 
and from $\g^{\times}$ if $K$ is ramified over $F$. 
Then $N$ is the maximal order $\mr$ in $K$ and $L = \mr + \g e + \g f$. 
Notice that 
$q \in \p[Z]$ if and only if $q\kappa[K^{\times}] = c\kappa[K^{\times}]$. 
By \cite[Theorem 3.5(iii)]{Y}, 
we can take $h$ and $\tau(q)$ in Case(3) as follows: 
\begin{gather*}
h = 
\begin{cases}
z + \pi^{m}e & \text{if $q \in \p[Z]$}, \\ 
z + \pi^{-m}q(1-s)e + \pi^{m}f & \text{if $q \not\in \p[Z]$, $2 \in \pe$, and $\xi = 0$}, \\ 
q\pi^{-m}e + \pi^{m}f & \text{otherwise}, 
\end{cases} \\ 
\tau(q) = 
\begin{cases}
[(\nu+d)/2] & \text{if $q \in \p[Z]$}, \\ 
[(\nu+d-1)/2] & \text{if $q \not\in \p[Z]$}.  
\end{cases}
\end{gather*}

Suppose that $q \in \p[Z]$. 
Then $h = z + \pi^{m}e \in N + \g e$ 
with $0 \le m \le [(\nu+d)/2]$ and a fixed element $z$ of $N[q]$. 
To take an explicit basis of $(Fz)^{\perp}$, 
we are going to solve the equation $c\kappa[z] = q$ with suitably chosen $c$. 
For such an element $z$ we set $N_{0} = (Fz)^{\perp} \cap N$, 
which is a $\g$-maximal lattice in the anisotropic space $(Fz)^{\perp} \cap Z$ of dimension $1$. 

Assume $\delta \in \pi \g^{\times}$. 
Then $K$ is ramified over $F$, 
$\mr = \g[\sqrt{\delta}]$, 
and $d = 2\kappa + 1$. 
Since $\kappa[\sqrt{\delta}] = -\delta$ is a prime element of $F$, 
put $q = (-\delta)^{\nu}q_{0}$ with $q_{0} \in \g^{\times}$. 
Because $c\kappa[K^{\times}] = q_{0}\kappa[K^{\times}]$ with $c$ above, 
we may take $q_{0}$ as $c$; 
namely, 
we can identify $(Z,\, \p)$ with $(K,\, q_{0}\kappa)$. 
Then $N = \mr$ and $z = \sqrt{\delta}^{\nu}$ satisfies $\p[z] = q_{0}\kappa[z] = q$, 
which is the required element. 
Thus $(Fz)^{\perp} \cap K$ is $F\sqrt{\delta}$ or $F$ 
according as $\nu$ is even or odd. 
If $\nu$ is even, 
then $N_{0} = \g \sqrt{\delta}$ is a $\g$-maximal lattice in $F\sqrt{\delta}$. 
Now, 
$\mr = \g \sqrt{\delta} + \g \supset N_{0} = \g \sqrt{\delta}$ and $z = \delta^{\nu/2}$. 
Hence by \LEM{l2}(4) we have 
$[M/L \cap W] = \g \cap 2q\pi^{-m}\delta^{-\nu/2}\g = \pe^{\nu/2+\kappa-m}$.  
If $\nu$ is odd, 
then $N_{0} = \g$. 
Since $\mr = \g + \g \sqrt{\delta} \supset N_{0} = \g$ and $z = \delta^{(\nu-1)/2}\sqrt{\delta}$, 
by \LEM{l2}(4) we have 
$[M/L \cap W] = \g \cap 2q\pi^{-m}\delta^{-(\nu-1)/2}\g = \pe^{(\nu+1)/2+\kappa-m}$.  

Assume that $\delta \in \g^{\times}$ and $2 \not\in \pe$. 
Then $K$ is unramified over $F$ and $\mr = \g[\sqrt{\delta}]$. 
Since $q\kappa[K^{\times}] = c\kappa[K^{\times}]$, 
it can be seen that 
$\nu$ is even if and only if $c \in \kappa[K^{\times}]$. 
Thus, 
if $q = \pi^{\nu}q_{0}$ with $q_{0} \in \g^{\times}$, 
we may take $c = q_{0}$ or $c = \pi q_{0}$ according as $\nu$ is even or odd. 
In this situation, 
$z = \pi^{[\nu/2]}$ satisfies $\p[z] = q$, 
so that $(Fz)^{\perp} \cap K = F\sqrt{\delta}$ and 
$N_{0} = \g \sqrt{\delta}$ is $\g$-maximal in $F\sqrt{\delta}$. 
From this we see that 
$[M/L \cap W] = \g \cap 2q\pi^{-m}\pi^{-[\nu/2]}\g = \pe^{[(\nu+1)/2]-m}$.  

Assume that $\delta \in \g^{\times}$, $2 \in \pe$, and $K$ is unramified over $F$. 
Applying \REF{35u} to $\delta$, 
we have $\delta \in (1 + \pi^{2\kappa}\g^{\times})\g^{\times 2}$. 
Put $\delta = (1 + \pi^{2\kappa}a)b^{-2}$ with $a,\, b \in \g^{\times}$, 
and also put $u = 2^{-1}(1 + b\sqrt{\delta}) \in \mr^{\times}$. 
Then we see that $\mr = \g[u]$. 
By the same argument as in the case $\delta \in \g^{\times}$ and $2 \not\in \pe$, 
we may put $c = q_{0}$ or $c = \pi q_{0}$ according as $\nu$ is even or odd. 
Then $z = \pi^{[\nu/2]} \in N[q]$, 
so that $(Fz)^{\perp} \cap K = F\sqrt{\delta}$ and $N_{0} = \g \sqrt{\delta}$. 
These $N_{0}$ and $N = \mr$ can be written as 
\begin{gather*}
N_{0} = \g
\begin{bmatrix}
   		1 & 0  \\
\end{bmatrix}
\begin{bmatrix}
   		1 & -2 \\ 
		0 & 1 \\
\end{bmatrix}
\begin{bmatrix}
   		1 \\
		u \\
\end{bmatrix}
,\qquad 
\mr = \g^{1}_{2}
\begin{bmatrix}
   		1 & -2 \\ 
		0 & 1 \\ 
\end{bmatrix}
\begin{bmatrix}
   		1 \\
		u \\  
\end{bmatrix}.
\end{gather*}
We have then $z = -b\pi^{[\nu/2]}\sqrt{\delta} + 2\pi^{[\nu/2]}u$, 
and hence $[M/L \cap W] = \g \cap q\pi^{-m-[\nu/2]}\g = \pe^{[(\nu+1)/2]-m}$.  

Assume that $\delta \in \g^{\times}$, $2 \in \pe$, and $K$ is ramified over $F$. 
Applying \REF{35r} to $\delta$, 
we have $\delta \in (1 + \pi^{2k+1}\g^{\times})\g^{\times 2}$ and $D_{K/F} = \pe^{2(\kappa-k)}$ 
with some $0 \le k < \kappa$. 
Put $\delta = (1 + \pi^{2k+1}a)b^{-2}$ with $a,\, b \in \g^{\times}$; 
put also $u = \pi^{-k}(1 + b\sqrt{\delta}) \in \mr$. 
This is a prime element of $K$ because $\kappa[u] = -\pi a$ is so of $F$. 
Hence $\mr = \g[u]$. 
Let $q = (-\pi a)^{\nu}q_{0}$ with $q_{0} \in \g^{\times}$. 
Since $c\kappa[K^{\times}] = q\kappa[K^{\times}]$ and $-\pi a \in \kappa[K^{\times}]$, 
we may put $c = q_{0}$. 
If $\nu$ is even, 
then $z = (-\pi a)^{\nu/2}$ belongs to $N[q]$, 
so that $(Fz)^{\perp} \cap K = F\sqrt{\delta}$ and 
$N_{0} = \g \sqrt{\delta}$ is $\g$-maximal in $F\sqrt{\delta}$. 
With a matrix 
$
\begin{bmatrix}
   		1 & -\pi^{k} \\ 
		0 & 1 \\
\end{bmatrix}
$ 
in $GL_{2}(\g)$, 
we have $\mr = \g \sqrt{\delta} + \g u$ and 
$z = -b(-\pi a)^{\nu/2}\sqrt{\delta} + \pi^{k}(-\pi a)^{\nu/2}u$. 
Thus \LEM{l2}(4) gives $[M/L \cap W] = 2q\pi^{-m-k-\nu/2}\g = \pe^{(\nu+d)/2-m}$. 
If $\nu$ is odd, 
then $z = (-\pi a)^{(\nu-1)/2}u \in N[q]$, 
so that $(Fz)^{\perp} \cap K = F(b\delta + \sqrt{\delta})$ and 
$N_{0} = \g \pi^{-k}(b\delta + \sqrt{\delta})$. 
With a matrix 
$
\begin{bmatrix}
   		\pi^{k+1}a & 1 \\ 
		1 & 0  
\end{bmatrix}
$ 
in $GL_{2}(\g)$, 
we have $\mr = \g \pi^{-k}(b\delta + \sqrt{\delta}) + \g$ and 
$z = (-\pi a)^{(\nu-1)/2}b\pi^{-k}(b\delta + \sqrt{\delta}) + (-a)^{(\nu+1)/2}\pi^{(\nu+1)/2+k}$. 
Thus \LEM{l2}(4) gives $[M/L \cap W] = 2q\pi^{-m-k-(\nu+1)/2}\g = \pe^{(\nu+d-1)/2-m}$. 
This finishes the proof in the case $q \in \p[Z]$. 
\\

Next suppose that $q \not\in \p[Z]$. 
Then $W$ is an anisotropic space of dimension $3$. 
If $2 \not\in \pe$ or $\xi = -1$, 
then $h = q\pi^{-m}e + \pi^{m}f \in \g e + \g f$ with $0 \le m \le [\nu/2]$. 
Since $q \not\in \p[Z]$, 
we have $[M/L \cap W] = \mathfrak{p}^{[\nu/2]-m}$ by \LEM{l2}(2). 

Assume that $2 \in \pe$ and $\xi = 0$. 
Let $\pi_{0}$ be a prime element of $F$ such that $\pi_{0} \in \kappa[K^{\times}]$, 
and put $q = \pi_{0}^{\nu}q_{0}$ with $q_{0} \in \g^{\times}$. 
Because $q \in q_{0} \kappa[K^{\times}]$ and 
the isomorphism class of $(Z,\, \p)$ is determined by $K$ and $c\kappa[K^{\times}]$, 
we may take $sq_{0}$ as such a $c$ for a fixed $s \in 1 + \pe^{d-1}$ so that $s \not\in \kappa[\mr^{\times}]$. 
Notice that equality \REF{cd} guarantees the existence of such an element $s$. 
Then we can identify $(Z,\, \p)$ with $(K,\, sq_{0}\kappa)$ and $N = \mr$. 
Let $\kappa[u] = \pi_{0}$ with $u \in \mr$, 
and put 
\begin{gather}
z = 
\begin{cases}
\pi_{0}^{\nu/2} & \text{if $\nu$ is even}, \\
\pi_{0}^{(\nu-1)/2}u & \text{if $\nu$ is odd}. 
\end{cases} \label{z1}
\end{gather}
Then we have $\mr = \g[u]$ and $\p[z] = sq$. 
In this situation \cite[Theorem 3.5]{Y} allows us to take 
$h = z + \pi^{-m}q(1-s)e + \pi^{m}f \in N + \g e + \g f$ 
with $0 \le m \le [(\nu+d-1)/2]$. 
It can be seen that 
\begin{gather}
W = U \oplus \left\{F\left(\frac{sq-q}{\pi^{2m}}e + f\right) + F\left(e - \frac{\pi^{m}}{2sq}z\right)\right\}, 
\quad U = (Fz)^{\perp} \cap K. 
\end{gather}
Put $g_{2} = \pi^{-2m}(sq-q)e + f$ and $g_{3} = e - (2sq)^{-1}\pi^{m}z$. 
Since $W$ is anisotropic, 
we need to construct a $\g$-maximal lattice $M$ in $W$, 
which will be done by employing the discriminant ideal. 
Once such an $M$ is obtained, 
we consider the analogy of \LEM{l2}(4). 
By \THM{di} the discriminant ideal of $W$ is given as follows: 
\begin{gather} 
[\widetilde{M}/M] = 
\begin{cases}
\pe^{\kappa + 1} & \text{if $\delta \in \pi \g^{\times}$ and $\nu$ is even, 
	or if $\delta \in \g^{\times}$ and $\nu$ is odd}, \\ 
\pe^{\kappa + 2} & \text{otherwise}. 
\end{cases} \label{m1}
\end{gather}

If $\delta \in \pi \g^{\times}$, 
we take $u$ of \REF{z1} to be $\sqrt{\delta}$ and $\pi_{0}$ to be $-\delta$. 
Then $U = Fg_{1}$ with $g_{1} = \sqrt{\delta}$ or $g_{1} = 1$ 
according as $\nu$ is even or odd. 
We set 
\begin{gather}
M = \g g_{1} + \g \pi^{-\mu_{2}}g_{2} + \g \pi^{-\mu_{3}}g_{3}, 
\end{gather}
where $\mu_{2} = [\ord(\p[g_{2}])/2]$ and $\mu_{3} = [\ord(\p[g_{3}])/2]$. 
This is a $\g$-lattice in $W$ such that $\p[M] \subset \g$. 
Observing the matrix representing $\p$ with respect to that $\g$-basis, 
we see that $[\widetilde{M}/M] = \pe^{\kappa+1}$ if $\nu$ is even 
and $[\widetilde{M}/M] = \pe^{\kappa+2}$ if $\nu$ is odd. 
Thus by \REF{m1} and \LEM{lid}(6), 
$M$ is $\g$-maximal in $W$. 
Since $\g g_{1} \subset \mr$, 
we can find a $\g$-basis $\{k_{1},\, k_{2}\}$ of $\mr$ such that 
$\g g_{1} = \g \varepsilon_{1}k_{1}$ with nonzero integer $\varepsilon_{1}$. 
Put $z = a_{1}k_{1} + a_{2}k_{2}$ with $a_{i} \in \g$; 
notice that $a_{2} \ne 0$. 
On the other hand, 
we have $\g e + \g f = \g g_{2} + \g e$. 
Then for 
$x = x_{1}\varepsilon_{1}k_{1} + x_{2}\pi^{-\mu_{2}}g_{2} + x_{3}\pi^{-\mu_{3}}g_{3} \in W$ 
with $x_{i} \in F$, 
we see that 
\begin{eqnarray*}
x \in L &\Longleftrightarrow& 
	\begin{cases}
	(x_{1}\varepsilon_{1} - (2sq\pi^{\mu_{3}})^{-1}\pi^{m}x_{3}a_{1})k_{1} 
		- (2sq\pi^{\mu_{3}})^{-1}\pi^{m}x_{3}a_{2}k_{2} \in \mr, & \\
	x_{2}\pi^{-\mu_{2}}g_{2} + x_{3}\pi^{-\mu_{3}}e \in \g g_{2} + \g e & \\
	\end{cases} \\ 
&\Longleftrightarrow& 
	[x_{1},\, x_{2},\, x_{3}] \in \g^{1}_{3}\alpha_{0}, 
\end{eqnarray*}
where $\alpha_{0}$ is the matrix in $GL_{3}(F)$ given by 
\begin{gather}
\begin{bmatrix}
   		\varepsilon_{1}^{-1} & 0 & 0 \\ 
   		0 & \pi^{\mu_{2}} & 0 \\ 
   		\pi^{m+\lambda}a_{1}(2sq\pi^{\mu_{3}}\varepsilon_{1})^{-1} & 0 & \pi^{\lambda} \\
\end{bmatrix},
\qquad \pe^{\lambda} = \frac{2sq\pi^{\mu_{3}}}{\pi^{m}a_{2}}\g \cap \pi^{\mu_{3}}\g. 
\end{gather} 
Thus the automorphism $\alpha$ of $W$, 
represented by $\alpha_{0}$ with respect to the basis 
$\{\varepsilon_{1}k_{1},\, \pi^{-\mu_{2}}g_{2},\, \pi^{-\mu_{3}}g_{3}\}$, 
gives a surjection of $M$ onto $L \cap W$. 
Therefore by \LEM{lid}(1) we find that 
\begin{gather}
[M/L \cap W] = \det(\alpha_{0})\g = \frac{\pi^{\mu_{2}}}{\varepsilon_{1}}\mathfrak{p}^{\lambda}. 
\label{i1}
\end{gather}
Hence our task is to find $\varepsilon_{1},\, a_{2}$, and $\pe^{\lambda}$. 
If $\nu$ is even, 
then we have 
\begin{gather*}
\mr = \g \sqrt{\delta} + \g,\quad \g g_{1} = \g \sqrt{\delta},\quad z = (-\delta)^{\nu/2}. 
\end{gather*}
Thus \REF{i1} shows that 
$[M/L \cap W] = \pi^{\mu_{2}}(\pe^{\nu/2+\kappa+\mu_{3}-m} \cap \pe^{\mu_{3}}) = \pe^{\nu/2+\kappa-m}$. 
If $\nu$ is odd, 
then 
\begin{gather*}
\mr = \g + \g \sqrt{\delta},\quad \g g_{1} = \g,\quad z = (-\delta)^{(\nu-1)/2}\sqrt{\delta}. 
\end{gather*}
Hence \REF{i1} gives 
$[M/L \cap W] = \pi^{\mu_{2}}(\pe^{(\nu+1)/2+\kappa+\mu_{3}-m} \cap \pe^{\mu_{3}}) = \pe^{(\nu-1)/2+\kappa-m}$. 

If $\delta \in \g^{\times}$, 
we take $u$ and $\pi_{0}$ of \REF{z1} to be $\pi^{-k}(1 + b\sqrt{\delta})$ and $-\pi a$, 
respectively. 
Here $a,\, b \in \g^{\times}$ and $0 \le k < \kappa$ are as in the case 
that $q \in \p[Z]$, $\delta \in \g^{\times}$, $2 \in \pe$, and $K$ is ramified over $F$. 
Then $(Fz)^{\perp} \cap K = Fg_{1}$ with $g_{1} = \sqrt{\delta}$ or $g_{1} = \pi^{-k}(b\delta + \sqrt{\delta})$ 
according as $\nu$ is even or odd. 
We set 
\begin{gather}
M = \g g_{1} + \g \pi^{-\mu_{2}}g_{2} + \g g, \quad 
g = \frac{1}{\pi^{k^{\prime}}}g_{1} + 
	\frac{2sq_{0}a^{[(\nu+1)/2]}}{b\pi^{\kappa+k^{\prime}+\mu_{3}}}g_{3}, 
\end{gather}
where $\mu_{i} = [\ord(\p[g_{i}])/2]$, 
$k^{\prime} = k$ if $\nu$ is even, 
and $k^{\prime} = k + 1$ if $\nu$ is odd. 
This is a $\g$-lattice in $W$ such that $\p[M] \subset \g$. 
Observing the matrix representing $\p$ with respect to that $\g$-basis, 
we see that $[\widetilde{M}/M] = \pe^{\kappa+2}$ if $\nu$ is even 
and $[\widetilde{M}/M] = \pe^{\kappa+1}$ if $\nu$ is odd. 
Thus $M$ is $\g$-maximal in $W$ by \REF{m1}. 
Since $\g g_{1} \subset \mr$, 
we can find a $\g$-basis $\{k_{1},\, k_{2}\}$ of $\mr$ such that 
$\g g_{1} = \g \varepsilon_{1}k_{1}$ with nonzero integer $\varepsilon_{1}$. 
Put $z = a_{1}k_{1} + a_{2}k_{2}$ with $a_{i} \in \g$. 
In the same way as in the case $\delta \in \pi\g^{\times}$ above, 
we obtain a surjection $\alpha$ of $M$ onto $L \cap W$, 
which is represented, 
with respect to the basis 
$\{\varepsilon_{1}k_{1},\, \pi^{-\mu_{2}}g_{2},\, g\}$, 
by the matrix 
\begin{gather}
\begin{bmatrix}
   		\varepsilon_{1}^{-1} & 0 & 0 \\ 
   		0 & \pi^{\mu_{2}} & 0 \\ 
   		\pi^{\lambda}c\varepsilon_{1}^{-1} & 0 & \pi^{\lambda} \\
\end{bmatrix},\qquad 
\pe^{\lambda} = \frac{2q\pi^{k^{\prime}+\mu_{3}}}{\pi^{m}a_{2}}\g \cap \pi^{k^{\prime}+\mu_{3}}\g 
\end{gather} 
with some $c \in F$. 
Thus we have 
\begin{gather}
[M/L \cap W] = \det(\alpha)\g = \frac{\pi^{\mu_{2}}}{\varepsilon_{1}}\mathfrak{p}^{\lambda}. 
\label{i2}
\end{gather}
If $\nu$ is even, 
then we have 
\begin{gather*}
\mr = \g \sqrt{\delta} + \g u,\quad \g g_{1} = \g \sqrt{\delta},\quad 
z = -(-\pi a)^{\nu/2}b\sqrt{\delta} + (-\pi a)^{\nu/2+k}u. 
\end{gather*}
We note that $c = -\pi^{-k}(1 + (-1)^{\nu/2}) \in \g$. 
Thus \REF{i2} shows that 
$[M/L \cap W] = \pi^{\mu_{2}}\g = \pe^{[(\nu+d-1)/2]-m}$. 
If $\nu$ is odd, 
then we have 
\begin{gather*}
\mr = \g g_{1} + \g,\quad z = (-\pi a)^{(\nu-1)/2}bg_{1} + (-a)^{(\nu+1)/2}\pi^{(\nu+1)/2+k}. 
\end{gather*}
We note that $c = -\pi^{-k-1}(1 + (-1)^{(\nu-1)/2}) \in \g$. 
Hence \REF{i2} gives 
$[M/L \cap W] = \pi^{\mu_{2}}\g = \pe^{[(\nu+d-1)/2]-m}$. 
This completes the proof in the case $q \not\in \p[Z]$. 

\subsection{Case $t = 3$}

Put $B = Q(\p)$. 
Then $B$ is a division algebra and 
$(Z,\, \p)$ can be identified with $(B^{\circ},\, -\delta \beta^{\circ})$. 
As noted in {\S}1.1, 
$B$ can be given by $\{J,\, c\}$ with a quadratic extension $J$ of $F$ 
and an element $c$ of $F^{\times}$ such that $c \not\in N_{J/F}(J^{\times})$. 
Put $K = F(\sqrt{\delta q})$, 
$\xi = \xi(\delta q)$, 
and $\varepsilon = \pi^{-\nu}q \in \g^{\times}$. 
Let $\mr$ be the maximal order in $K$. 
Notice that 
$q \in \p[Z]$ if and only if $K \ne F$. 
By \cite[Theorem 3.5(iv)]{Y}, 
we can take $h$ and $\tau(q)$ in Case(3) as follows: 
\begin{gather*}
h = 
\begin{cases}
z + \pi^{m}e & \text{if $q \in \p[Z]$}, \\ 
z + \pi^{-m}q(1-s)e + \pi^{m}f & \text{if $q \not\in \p[Z]$ and $2 \in \pe$}, \\ 
q\pi^{-m}e + \pi^{m}f & \text{if $q \not\in \p[Z]$ and $2 \not\in \pe$}, 
\end{cases}  
\end{gather*}
$\tau(q) =$ 
\begin{gather*}
\begin{cases}
\kappa + [\nu/2] & \text{if $q \not\in \p[Z]$}, \\ 
(\ord(\delta^{-1}q)+1)/2 & \text{if $q \in \p[Z]$ and $\ord(\delta q)$ is odd}, \\ 
\kappa + 1 + (\ord(\delta^{-1}q)-d)/2 & 
	\text{if $q \in \p[Z]$, $\ord(\delta q)$ is even, and $\xi = 0$}, \\ 
\kappa + \ord(\delta q)/2 & \text{otherwise}.  
\end{cases}
\end{gather*}

Suppose that $q \in \p[Z]$. 
Then $h = z + \pi^{m}e \in N + \g e$ 
with $0 \le m \le \tau(q)$ and a fixed element $z$ of $N[q]$. 
To take a basis of $(Fz)^{\perp}$, 
we solve the equation $-\delta \beta^{\circ}[z] = q$ by embedding $K$ into $B$. 
Namely, 
choosing $c \in F^{\times}$ so that $c \not\in \kappa[K^{\times}]$, 
we have $B = \{K,\, c\}$ as noted above. 
We may assume that $c = \pi$ if $\xi = -1$ and $c \in \g^{\times}$ if $\xi = 0$. 
Then $z = \sqrt{\delta^{-1}q}$ belongs to $K \cap B^{\circ}$ and satisfies $\p[z] = q$. 
Hence we see that 
\begin{gather*}
B^{\circ} = F\sqrt{\delta q} \oplus K\omega, \qquad 
W = K\omega \oplus (Fe + Ff_{0}) 
\end{gather*}
with $f_{0}$ in \REF{f0} and $\omega \in B$ in \REF{om} such that $\omega^{2} = c$. 
We note that $K\omega$ is an anisotropic space of dimension $2$ and 
isomorphic to $(K,\, \delta c\kappa)$ by $x\omega \mapsto x$. 
Set $N_{0} = K\omega \cap N$, 
which is a $\g$-maximal lattice in $K\omega$. 
We are going to find $\g$-bases of $N$ and $N_{0}$. 
In view of the uniqueness of $N$ and $N_{0}$, 
it is sufficient to construct a $\g$-maximal lattice in $B^{\circ}$ and also in $K\omega$, 
which will be done by employing the discriminant ideals of $\p$ and $\psi$.  

Assume that $\delta \in \g^{\times}$ and $\nu$ is odd. 
Also assume $2 \in \pe$. 
Then $K = F(\sqrt{\delta \varepsilon \pi})$ is ramified over $F$ and $d = 2\kappa + 1$; 
put $u = \sqrt{\delta \varepsilon \pi}$. 
The discriminant ideal of $\p$ (resp.\ $\psi$) is $\pe^{\kappa+2}$ (resp.\ $\pe^{2\kappa+1}$) by \THM{di}. 
To obtain $N$ in $B^{\circ}$, 
we take 
\begin{gather}
c \in 1 + \pe^{d-1}\quad \text{such that}\quad c \not\in \kappa[\mr^{\times}], \label{c}
\end{gather}
and identify $B$ with $\{K,\, c\}$. 
Equality \REF{cd} guarantees the existence of such an element $c$. 
Then $g = \pi^{-\kappa}(u + u\omega)$ belongs to $B^{\circ}$ and satisfies $\p[g] \in \pi\g^{\times}$. 
Set $N = \g u + \g \omega + \g g$. 
This is a $\g$-lattice in $B^{\circ}$ such that $\p[N] \subset \g$. 
Observing the matrix representing $\p$ with respect to that $\g$-basis, 
we see that $[\widetilde{N}/N] = \pe^{\kappa+2}$. 
Thus $N$ is $\g$-maximal in $B^{\circ}$ by \LEM{lid}(6). 
To obtain $N_{0}$ in $K\omega$, 
we consider the maximal order $\mr = \g[u]$ in $K$. 
Because $\delta c \in \g^{\times}$, 
this is $\g$-maximal with respect to $\delta c\kappa$; 
hence $N_{0} = \g \omega + \g u\omega$ is maximal in $K\omega$ with respect to $-\delta \beta^{\circ}$. 
(In fact, we have $[\widetilde{N_{0}}/N_{0}] = \pe^{2\kappa+1}$.) 
To apply \LEM{l2}(4), 
we take a $\g$-basis $\{k_{i}\}_{i=1}^{3}$ of $N$ determined by 
\begin{gather*}
\begin{bmatrix}
   	k_{1} \\ 
   	k_{2} \\ 
   	k_{3} \\
\end{bmatrix}
= 
\begin{bmatrix}
   	1 & 0 & \pi^{\kappa} \\ 
   	0 & 1 & 0 \\ 
   	0 & 0 & 1 \\
\end{bmatrix}^{-1}
\begin{bmatrix}
   	u \\ 
   	\omega \\ 
   	g \\
\end{bmatrix}.
\end{gather*} 
Then $N_{0} = \g k_{1} + \g k_{2}$ and 
$z = \delta^{-1}\pi^{(\nu-1)/2}(k_{1} + \pi^{\kappa}k_{3})$. 
Hence we have 
$[M/L \cap W] = \g \cap \pi^{-(\nu-1)/2-\kappa-m}2q\g = \pe^{(\nu+1)/2-m}$. 
If $2 \not\in \pe$, 
then $N = \g u + \g \omega + \g u\omega$ and $N_{0} = \g \omega + \g u\omega$ 
are both $\g$-maximal with $u = \sqrt{\delta \varepsilon \pi}$, 
and so $[M/L \cap W] = \pe^{(\nu+1)/2-m}$. 

Assume that $\delta \in \pi\g^{\times}$ and $\nu$ is even. 
Also assume $2 \in \pe$. 
Then $K = F(\sqrt{\delta_{1} \varepsilon \pi})$ is ramified over $F$ and $d = 2\kappa + 1$, 
where $\delta_{1} = \pi^{-1}\delta \in \g^{\times}$; 
put $u = \sqrt{\delta_{1} \varepsilon \pi}$. 
The discriminant ideal of $\p$ (resp.\ $\psi$) is $\pe^{\kappa+1}$ (resp.\ $\pe^{2\kappa+1}$) by \THM{di}. 
To obtain $N$ in $B^{\circ}$, 
we take $c$ as in \REF{c} and identify $B$ with $\{K,\, c\}$. 
Set $g = \pi^{-\kappa}(u^{-1} + u^{-1}\omega)$ and $N = \g u^{-1} + \g \omega + \g g$, 
which is a $\g$-lattice in $B^{\circ}$ such that $\p[N] \subset \g$. 
Observing the matrix representing $\p$ with respect to the $\g$-basis, 
we have $[\widetilde{N}/N] = \pe^{\kappa+1}$, 
so that $N$ is $\g$-maximal in $B^{\circ}$. 
We further set $N_{0} = \g \omega + \g u^{-1}\omega$. 
Since $[\widetilde{N_{0}}/N_{0}] = \pe^{2\kappa+1}$, 
this is $\g$-maximal in $K\omega$. 
In a similar way to the case $\delta \in \g^{\times}$ and $\nu$ is odd, 
we obtain a $\g$-basis $\{k_{i}\}_{i=1}^{3}$ of $N$ such that 
$N_{0} = \g k_{1} + \g k_{2}$ and 
$z = \varepsilon \pi^{\nu/2}(k_{1} + \pi^{\kappa}k_{3})$. 
Hence we have 
$[M/L \cap W] = \g \cap \pi^{-\nu/2-\kappa-m}2q\g = \pe^{\nu/2-m}$. 
If $2 \not\in \pe$, 
then $N = \g u^{-1} + \g \omega + \g u^{-1}\omega$ and 
$N_{0} = \g \omega + \g u^{-1}\omega$ are $\g$-maximal with $u = \sqrt{\delta_{1} \varepsilon \pi}$, 
and so we have $[M/L \cap W] = \pe^{\nu/2-m}$. 
This completes the proof of the case where $q \in \p[Z]$ and $\ord(\delta q)$ is odd. 

Assume that $\delta \in \g^{\times}$, $\nu$ is even, and $K$ is ramified over $F$. 
Then $2 \in \pe$ and 
the discriminant ideal of $\p$ (resp.\ $\psi$) is $\pe^{\kappa+2}$ (resp.\ $\pe^{d}$) by \THM{di}.  
Applying \REF{35r} to $\delta \varepsilon \in \g^{\times}$, 
we have $\delta \varepsilon \in (1 + \pi^{2k+1}\g^{\times})\g^{\times 2}$ and $D_{K/F} = \pe^{2(\kappa-k)}$ 
with some $0 \le k < \kappa$. 
Hence we may put $\delta \varepsilon = (1 + \pi^{2k+1}a)b^{-2}$ with $a,\, b \in \g^{\times}$. 
Since $u = \pi^{-k}(1 + b\sqrt{\delta \varepsilon})$ is a prime element of $K$, 
$\g[u]$ is the maximal order $\mr$ in $K$. 
To obtain $N$ in $B^{\circ}$, 
we take $c$ as in \REF{c} and identify $B$ with $\{K,\, c\}$. 
Then we find $g = 2^{-1}\pi^{k+1}\sqrt{\delta \varepsilon}(1 + \omega)$ 
in $B^{\circ}$ such that $\p[g] \in \pi\g^{\times}$. 
Set $N = \g \sqrt{\delta \varepsilon} + \g u\omega + \g g$ in $B^{\circ}$. 
This satisfies $\p[N] \subset \g$ and $[\widetilde{N}/N] = \pe^{\kappa+2}$. 
Hence $N$ is $\g$-maximal in $B^{\circ}$. 
Further since $\delta c \in \g^{\times}$, 
$\mr$ is $\g$-maximal in $K$ with respect to $\delta c\kappa$, 
and hence $N_{0} = \mr \omega$ is maximal in $K\omega$ with respect to $-\delta \beta^{\circ}$. 
To apply \LEM{l2}(4), 
we find a $\g$-basis $\{k_{i}\}_{i=1}^{3}$ of $N$ determined by 
\begin{gather*}
\begin{bmatrix}
   	k_{1} \\ 
   	k_{2} \\ 
   	k_{3} \\
\end{bmatrix}
= 
\begin{bmatrix}
   	1 & -b^{-1}\pi^{k} & 2\pi^{-k-1} \\ 
   	0 & 1 & 0 \\ 
   	0 & 0 & 1 \\
\end{bmatrix}^{-1}
\begin{bmatrix}
   	\sqrt{\delta \varepsilon} \\ 
   	u\omega \\ 
   	g \\
\end{bmatrix}.
\end{gather*} 
Then $N_{0} = \g k_{1} + \g k_{2}$ and 
$z = \delta^{-1}\pi^{\nu/2}(k_{1} - b^{-1}\pi^{\kappa}k_{2} + 2\pi^{-k-1}k_{3})$. 
Hence we obtain 
$[M/L \cap W] = \g \cap \pi^{-\nu/2-\kappa+k+1-m}2q\g = \pe^{\nu/2+k+1-m}$. 

Assume that $\delta \in \pi\g^{\times}$, $\nu$ is odd, and $K$ is ramified over $F$. 
Then $2 \in \pe$. 
A similar argument gives $\g$-bases of $N$ and $N_{0}$ as follows: 
\begin{gather*}
N = \g \sqrt{\delta_{1} \varepsilon} + \g u^{-1}\omega + 
	\g 2^{-1}\pi^{k}\sqrt{\delta_{1} \varepsilon}(1 + \omega), \quad 
N_{0} = \g \omega + \g u^{-1}\omega, 
\end{gather*}
where $\delta_{1} = \pi^{-1}\delta \in \g^{\times}$, 
$\delta_{1} \varepsilon = (1 + \pi^{2k+1}a)b^{-2}$ with $a,\, b \in \g^{\times}$, 
$0 \le k < \kappa$, 
and $u = \pi^{-k}(1 + b\sqrt{\delta_{1} \varepsilon})$. 
The expression for $N_{0}$ can be verified from $[\widetilde{N_{0}}/N_{0}] = \pe^{2\kappa-2k}$. 
Then we find a basis $\{k_{i}\}_{i=1}^{3}$ of $N$ determined by 
\begin{gather*}
\begin{bmatrix}
   	k_{1} \\ 
   	k_{2} \\ 
   	k_{3} \\
\end{bmatrix}
= 
\begin{bmatrix}
   	1 & -b^{-1}a\pi^{k+1} & 2\pi^{-k} \\ 
   	0 & 1 & 0 \\ 
   	0 & 0 & 1 \\
\end{bmatrix}^{-1}
\begin{bmatrix}
   	\sqrt{\delta_{1} \varepsilon} \\ 
   	u^{-1}\omega \\ 
   	2^{-1}\pi^{k}\sqrt{\delta_{1} \varepsilon}(1 + \omega) \\
\end{bmatrix}.
\end{gather*}
Observe that $N_{0} = \g k_{1} + \g k_{2}$. 
Since the coefficient of $z$ associated with $k_{3}$ is $\delta_{1}^{-1}2\pi^{(\nu-1)/2-k}$, 
\LEM{l2}(4) shows $[M/L \cap W] = \pe^{(\nu+1)/2+k-m}$. 
This finishes the proof in the case where $q \in \p[Z]$, $\ord(\delta q)$ is even, and $\xi = 0$. 

Assume that $K$ is unramified over $F$. 
Then $B = \{K,\, \pi\}$ and $\ord(\delta q)$ is even. 
Put $v = \sqrt{\delta \varepsilon}$ if $\delta \in \g^{\times}$ and 
$v = \sqrt{\delta_{1} \varepsilon}$ if $\delta = \pi \delta_{1} \in \pi\g^{\times}$. 
Clearly $K = F(v)$. 
Further if $2 \in \pe$, 
we may put $v^{2} = (1 + \pi^{2\kappa}a)b^{-2}$ with $a,\, b \in \g^{\times}$ 
by applying \REF{35u} to $v^{2} \in \g^{\times}$. 
Then we see that $\mr = \g[v]$ if $2 \not\in \pe$ and $\mr = \g[u]$ if $2 \in \pe$, 
where $u = \pi^{-\kappa}(1 + bv)$. 
Since $B$ is a division algebra and $K$ is unramified over $F$, 
the maximal lattice in $B^{\circ}$ with respect to $-\delta \beta^{\circ}$ can be given by 
\begin{gather}
N = 
\begin{cases}
\g v + \mr \omega & \text{if $\delta \in \g^{\times}$}, \\
\g v + \mr \omega^{-1} & \text{if $\delta \in \pi\g^{\times}$}. 
\end{cases} \label{un}
\end{gather}
This can be found in the proof of \cite[Theorem 6.2]{cq}. 
Since $Q(\psi) = B$ if $\delta \in \g^{\times}$ by \THM{t1}, 
the discriminant ideal of $\psi$ is $\pe^{2}$ or $\g$ 
according as $\delta \in \g^{\times}$ or $\delta \in \pi\g^{\times}$ by \THM{di}. 
If $\delta \in \g^{\times}$, 
then $\mr$ is $\g$-maximal in $K$ with respect to $\delta \pi \kappa$, 
because of $\delta \pi \in \pi \g^{\times}$. 
Hence $N_{0} = \mr \omega$ is $\g$-maximal in $K\omega$ with respect to $-\delta \beta^{\circ}$. 
From this together with \REF{un} and $z = \delta^{-1}\pi^{\nu/2}v$, 
we see that $[M/L \cap W] = \pe^{\nu/2+\kappa-m}$. 
If $\delta \in \pi\g^{\times}$, 
we have $N_{0} = \g \omega^{-1} + \g v\omega^{-1}$ if $2 \not\in \pe$ 
and $N_{0} = \g \omega^{-1} + \g u\omega^{-1}$ if $2 \in \pe$. 
This is because $\p[N_{0}] \subset \g$ and $[\widetilde{N_{0}}/N_{0}] = \g$. 
Combining this with \REF{un} and $z = \delta^{-1}\pi^{(\nu+1)/2}v$, 
we obtain $[M/L \cap W] = \pe^{(\nu+1)/2+\kappa-m}$. 
This completes the proof in the case $q \in \p[Z]$. 
\\

Next suppose that $q \not\in \p[Z]$. 
Then $W$ is an anisotropic space of dimension $4$. 
If $2 \not\in \pe$, 
then $h = q\pi^{-m}e + \pi^{m}f \in \g e + \g f$ with $0 \le m \le [\nu/2]$. 
Since $q \not\in \p[Z]$, 
we have $[M/L \cap W] = \mathfrak{p}^{[\nu/2]-m}$ by \LEM{l2}(2). 

Assume that $2 \in \pe$. 
Since $q \not\in \p[Z]$, 
$\delta q$ is the square of an element of $F^{\times}$. 
Fix an unit element $s$ so that 
\begin{gather}
s \in 1 + \pi^{2\kappa}\g^{\times},\qquad s \not\in \g^{\times 2}. \label{s}
\end{gather}
\cite[Lemma 3.2(1)]{99b} guarantees the existence of such an element $s$. 
Set $J = F(\sqrt{s})$ and $u = \pi^{-\kappa}(1 + \sqrt{s})$. 
Then $J$ is an unramified quadratic extension of $F$ 
and $\g[u]$ is the maximal order $\mr_{1}$ in $J$. 
We may identify $B$ with $\{J,\, \pi\}$. 
Then in the same way as in \REF{un} we see that 
\begin{gather*}
B^{\circ} = F\sqrt{s} \oplus J\omega, \qquad \omega^{2} = \pi, \\
N = 
\begin{cases}
\g \sqrt{s} + \mr_{1} \omega & \text{if $\delta \in \g^{\times}$}, \\
\g \sqrt{s} + \mr_{1} \omega^{-1} & \text{if $\delta \in \pi\g^{\times}$}. 
\end{cases} 
\end{gather*}
Now, 
let $\delta^{-1}q = b^{2}$ with $b \in F^{\times}$ and put $z = b\sqrt{s}$ in $B^{\circ}$; 
then $b \in \g$ and $\p[z] = -\delta N_{J/F}(b\sqrt{s}) = sq$, 
so that $z \in N[sq]$. 
In this situation \cite[Theorem 3.5]{Y} allows us to take 
$h = z + \pi^{-m}q(1-s)e + \pi^{m}f \in N + \g e + \g f$ 
with $0 \le m \le [\nu/2] + \kappa$. 
It can be seen that 
\begin{gather*}
W = J\omega \oplus 
	\left\{F\left(\frac{sq-q}{\pi^{2m}}e + f\right) + F\left(e - \frac{\pi^{m}}{2sq}z\right)\right\}. 
\end{gather*}
Put $g_{3} = \pi^{-2m}(sq-q)e + f$ and $g_{4} = e - (2sq)^{-1}\pi^{m}z$. 
Then $\g e + \g f = \g g_{3} + \g e$ as observed in {\S}3.2. 
Since the discriminant ideal of $W$ is $\pe^{2}$ by \THM{di}, 
we then find a $\g$-maximal lattice in $W$ given by 
\begin{gather*}
M = 
\begin{cases}
\mr_{1} \omega + \g \pi^{-[\mu/2]}g_{3} + \g \pi^{-[-\mu/2]}g_{4} & \text{if $\delta \in \g^{\times}$}, \\
\mr_{1} \omega^{-1} + \g \pi^{-[\mu/2]}g_{3} + \g \pi^{-[-\mu/2]}g_{4} & \text{if $\delta \in \pi\g^{\times}$}, 
\end{cases} 
\end{gather*}
where $\mu = \nu+2\kappa-2m$. 
Put $g_{1} = \omega$ and $g_{2} = u\omega$ if $\delta \in \g^{\times}$, 
and put $g_{1} = \omega^{-1}$ and $g_{2} = u\omega^{-1}$ if $\delta \in \pi\g^{\times}$. 
Then for $x = x_{1}g_{1} + x_{2}g_{2} + x_{3}\pi^{-[\mu/2]}g_{3} + x_{4}\pi^{-[-\mu/2]}g_{4} \in W$, 
$x$ belongs to $L$ if and only if 
$[x_{1}\, \cdots\, x_{4}] \in \g^{1}_{4}\mathrm{diag}[1,\, 1,\, \pi^{[\mu/2]},\, 1]$. 
This shows that 
$L \cap W = \mr_{1}g_{1} + \g g_{3} + \g \pi^{-[-\mu/2]}g_{4}$. 
Therefore we obtain 
$[M/L \cap W] = \mathfrak{p}^{[\nu/2]+\kappa-m}$. 
This completes the proof in the case $q \not\in \p[Z]$, 
and hence $t = 3$. 

\subsection{Case $t = 4$}

Since $F(\sqrt{\delta}) = F$, 
the discriminant field $K$ of $\psi$ is $F(\sqrt{-q})$. 
Put $\varepsilon = \pi^{-\nu}q \in \g^{\times}$ and $B = Q(\p)$. 
Then $B$ is a division algebra and $(Z,\, \p)$ can be identified with $(B,\, \beta)$. 
Because $\beta[B^{\times}] = F^{\times}$, 
there exists an element $z$ in $N$ such that $\beta[z] = q$ for any $0 \ne q \in \g$. 
For such a $z$ we can take $h = z + \pi^{m}e$ with $0 \le m \le [(\nu+1)/2]$ 
by \cite[Theorem 3.5(v)]{Y}. 
To take an explicit $z \in N[q]$, 
we solve the equation $\beta[z] = q$ by embedding $K$ into $B$ if $K \ne F$ 
and by identifying $(B,\, \beta)$ with $(B,\, -\beta)$ if $K = F$. 
\\

Suppose that $K \ne F$. 
Put $\xi = \xi(-q)$ and 
$v = \sqrt{-\varepsilon}$ or $v = \sqrt{-\pi\varepsilon}$ 
according as $\nu$ is even or odd. 
Let $c$ be an element of $\g^{\times}$ such that $c \not\in \kappa[K^{\times}]$ if $\xi = 0$, 
and $c = \pi$ if $\xi = -1$. 
Then $B = \{K,\, c\}$. 
Take $z = \sqrt{-q}$ in $K$, 
which satisfies $\beta[z] = \kappa[z] = q$, 
so that $z \in N[q]$. 
Hence we have 
\begin{gather*}
B = K \oplus K\omega, \qquad 
W = (F \oplus K\omega) \oplus (Fe + Ff_{0}) 
\end{gather*}
with $f_{0}$ in \REF{f0} and $\omega$ in \REF{om} such that $\omega^{2} = c$. 
We note that $F \oplus K\omega$ is an anisotropic space of dimension $3$ and 
isomorphic to $(B^{\circ},\, -v^{2}\beta^{\circ})$ by $-xv \mapsto x$. 
Set $N_{0} = (F \oplus K\omega) \cap N$, 
which is a $\g$-maximal lattice in $F \oplus K\omega$. 
We are going to find $\g$-bases of $N$ and $N_{0}$, 
namely, 
to construct a $\g$-maximal lattice in $B$ and also in $F \oplus K\omega$ 
by employing the discriminant ideals of $\p$ and $\psi$.
\THM{di} shows that 
\begin{gather*} 
[\widetilde{N}/N] = \pe^{2}, \qquad 
[\widetilde{N_{0}}/N_{0}] = 
\begin{cases}
\pe^{\kappa + 2} & \text{if $\nu$ is even}, \\ 
\pe^{\kappa + 1} & \text{if $\nu$ is odd}. 
\end{cases} 
\end{gather*}

Assume that $\nu$ is even and $K$ is ramified over $F$. 
Then $2 \in \pe$. 
By \REF{35r} we can put $-\varepsilon = (1 + \pi^{2k+1}a)b^{-2}$ and $D_{K/F} = \pe^{2(\kappa-k)}$ 
with some $0 \le k < \kappa$ and $a,\, b \in \g^{\times}$; 
the maximal order $\mr$ of $K$ is given by $\g[u]$, 
where $u = \pi^{-k}(1 + bv)$. 
Further, 
we may take $c \in 1 + \pe^{d-1}$ such that $c \not\in \kappa[\mr^{\times}]$ by \REF{cd}. 
Then there are four elements of $B$ given by 
\begin{gather}
g_{1} = 1,\quad 
g_{2} = u\omega = \pi^{-k}(1 + bv)\omega, \nonumber \\
g_{3} = \pi^{k+1-\kappa}(1 + \omega),\quad 
g_{4} = \pi^{-\kappa}(1 + bv + \omega + bv\omega). 
\end{gather}
We set $N = \sum_{i=1}^{4} \g g_{i}$. 
This is a $\g$-lattice in $B$ such that $\p[N] \subset \g$ and $[\widetilde{N}/N] = \pe^{2}$. 
Hence $N$ is $\g$-maximal. 
Moreover we have $(F \oplus K\omega) \cap N = \sum_{i=1}^{3} \g g_{i}$, 
which is the required lattice $N_{0}$ in $F \oplus K\omega$. 
Some calculations show $z = -b^{-1}\pi^{\nu/2}(g_{1} + \pi^{k}g_{2} - \pi^{\kappa}g_{4})$. 
Thus \LEM{l2}(4) gives $[M/L \cap W] = \mathfrak{p}^{\nu/2-m}$. 

Assume that $\nu$ is odd. 
Then $K$ is ramified over $F$ and $D_{K/F} = \pe^{2\kappa+1}$. 
If $2 \in \pe$, 
we take $c$ as in the above case and identify $B$ with $\{K,\, c\}$. 
It can be verified that 
$N = \g + \g v + \g \pi^{-\kappa}(1 + \omega) + \g \pi^{-\kappa}v(1 + \omega)$ and 
$N_{0} = \g + \g \pi^{-\kappa}(1 + \omega) + \g v\omega$ are $\g$-maximal in $B$ and $F \oplus K\omega$, 
respectively. 
Taking a basis of $N$ given by 
\begin{gather*}
\begin{bmatrix}
   	k_{1} \\ 
   	k_{2} \\ 
   	k_{3} \\
  	k_{4} \\
\end{bmatrix}
= 
\begin{bmatrix}
   	1 & 0 & 0 & 0 \\ 
   	0 & 1 & 0 & 0 \\ 
   	0 & 0 & -1 & \pi^{\kappa} \\
   	0 & 0 & 0 & 1 \\
\end{bmatrix}^{-1}
\begin{bmatrix}
   	1 \\ 
   	\pi^{-\kappa}(1 + \omega) \\ 
   	v \\ 
   	\pi^{-\kappa}v(1 + \omega) \\
\end{bmatrix},
\end{gather*}
we have $N_{0} = \sum_{i=1}^{3} \g k_{i}$ and $z = \pi^{[\nu/2]}(-k_{3} + \pi^{\kappa}k_{4})$. 
Hence we obtain $[M/L \cap W] = \mathfrak{p}^{(\nu+1)/2-m}$. 
If $2 \not\in \pe$, 
it can easily be seen that 
$N = \g + \g \omega + \g v\omega + \g v$ and $N_{0} = \g + \g \omega + \g v\omega$ 
are respectively $\g$-maximal in $B$ and $F \oplus K\omega$, 
so that we have $[M/L \cap W] = \mathfrak{p}^{(\nu+1)/2-m}$. 

Assume that $\nu$ is even and $K$ is unramified over $F$. 
Let $\mr$ be the maximal order of $K$. 
Because $B$ is a division algebra and $K$ is unramified over $F$, 
the maximal order $N$ in $B$ can be given by 
\begin{gather}
B = K \oplus K\omega,\quad 
N = \mr + \mr \omega,\quad 
\omega^{2} = \pi. \label{un2}
\end{gather}
(See\ \cite[Theorem\ 5.13]{04},\ for example.) 
Then we have $N_{0} = (F \oplus K\omega) \cap N = \g + \mr \omega$. 
If $2 \not\in \pe$, 
then $\mr = \g[v]$. 
Since $z = \pi^{\nu/2}v$, 
we have $[M/L \cap W] = \mathfrak{p}^{\nu/2-m}$. 
If $2 \in \pe$, 
we may put $-\varepsilon = 1 + \pi^{2\kappa}a$ with $a \in \g^{\times}$ by \REF{35u}. 
Then $\mr = \g[u]$, 
where $u = \pi^{-\kappa}(1 + v)$. 
Since $z = \pi^{\nu/2}(\pi^{\kappa}u - 1)$, 
we obtain $[M/L \cap W] = \mathfrak{p}^{\nu/2-m}$. 
\\

Suppose that $K = F$, 
that is, 
$-q \in F^{\times 2}$. 
We may identify $(Z,\, \p)$ with $(B,\, -\beta)$. 
Further, 
we can take a division algebra $B = \{J,\, \pi\}$ 
with an unramified quadratic extension $J = F(\sqrt{s})$ of $F$. 
Here $s$ is a fixed unit element; 
if $2 \in \pe$, 
we choose $s$ as in \REF{s} and put $u = \pi^{-\kappa}(1 + \sqrt{s})$. 
Let $\mr_{1}$ be the maximal order in $J$. 
Then since $N$ coincides with the maximal order in $B$, 
we have $N = \mr_{1} + \mr_{1}\omega$ with $\omega^{2} = \pi$ by the same reason as in \REF{un2}. 
Now take $z = \sqrt{-q} \in F$, 
which satisfies $\p[z] = -\beta[z] = q$. 
Then $(Fz)^{\perp} \cap B = F\sqrt{s} \oplus J\omega = B^{\circ}$ and $\psi = -\beta^{\circ}$. 
Since $(Fz)^{\perp} \cap N$ is the $\g$-maximal lattice $N_{0}$ in $(Fz)^{\perp} \cap B$, 
we have $N_{0} = \g \sqrt{s} + \mr_{1}\omega$. 
If $2 \in \pe$, 
taking a basis of $N$ given by 
\begin{gather*}
\begin{bmatrix}
   	k_{1} \\ 
   	k_{2} \\ 
   	k_{3} \\
  	k_{4} \\
\end{bmatrix}
= 
\begin{bmatrix}
   	1 & 0 & 0 & 0 \\ 
   	0 & 1 & 0 & 0 \\ 
   	0 & 0 & -1 & \pi^{\kappa} \\
   	0 & 0 & 0 & 1 \\
\end{bmatrix}^{-1}
\begin{bmatrix}
   	\omega \\ 
   	u\omega \\ 
   	1 \\ 
   	u \\
\end{bmatrix},
\end{gather*}
we have $N_{0} = \sum_{i=1}^{3} \g k_{i}$ and $z = \sqrt{-q}(-k_{3} + \pi^{\kappa}k_{4})$. 
Hence we obtain $[M/L \cap W] = \mathfrak{p}^{\nu/2-m}$. 
If $2 \not\in \pe$, 
then $\mr_{1} = \g[\sqrt{s}]$, 
and thus \LEM{l2}(4) gives 
$[M/L \cap W] = \pi^{-m}z^{-1}2q\g = \mathfrak{p}^{\nu/2-m}$. 
This finishes the proof in the case $t = 4$. 
\\

In all cases needed in the proof so far, 
we can find a nonnegative integer $\lambda$ such that 
$[M/L \cap W] = \pe^{\lambda-m}$ if $2\p(h,\, L) = \pe^{m}$. 
That $\lambda$ is determined by the isomorphism class of $(V,\, \p)$ and $q$. 
This completes the proof of \THM{t2}. 
\\

Let us here state $\tau(q)$ of \cite[Theorem 3.5]{Y} in the case $r = 0$, 
which was treated in Case(2) of {\S}2.2: 

$\tau(q) =$ 
\begin{gather}
\begin{cases}
\kappa + (\nu+\ord(\delta))/2 & \text{if $t = 1$}, \\ 
[(\nu+d)/2] & \text{if $t = 2$}, \\ 
(\ord(\delta^{-1}q)+1)/2 
	& \text{if $t = 3$ and $\ord(\delta q)\not\in2\mathbf{Z}$}, \\ 
\kappa + 1 + (\ord(\delta^{-1}q)-d)/2 
	& \text{if $t = 3$, $\ord(\delta q)\in2\mathbf{Z}$, and $\xi(\delta q) = 0$}, \\ 
\kappa + \ord(\delta q)/2 
	& \text{if $t = 3$, $\ord(\delta q)\in2\mathbf{Z}$, and $\xi(\delta q) = -1$}, \\ 
[(\nu+1)/2] & \text{if $t = 4$}. 
\end{cases} \label{corid}
\end{gather}
Clearly this coincides with $\lambda(q)$ in the case $r = 0$, 
that is, 
in the anisotropic case. 
That is true even if $r > 0$ and $q \in \p[Z]$. 
We note that \REF{corid} can be derived by using explicit solution of $\p[z] = q$ and 
basis of $\g$-maximal lattice $N$ in the case $q \in \p[Z]$ of the proof in Case (3). 
\\

To prove the corollary below, 
we need the discriminant ideals of both anisotropic spaces and 
its orthogonal complement determined by $0 \ne q \in \g$. 
Those ideals can be derived by applying \THM{di} to each of the cases 
that treated in the proof of \THM{t2} under the assumption $q \in \p[Z]$. 
In this way $[\widetilde{L}/L]$ and $[\widetilde{M}/M]$ in question are given as follows: 
\begin{gather}
[\widetilde{L}/L] = 
\begin{cases}
\pe^{2\kappa +1} & \text{if $t = 2$, $\delta \in \pi\g^{\times}$, and $\nu \in 2\mathbf{Z}$,} \\ 
 & \text{or $t = 2$, $\delta \in \pi\g^{\times}$, and $\nu \not\in 2\mathbf{Z}$,} \\ 
\pe^{d} & \text{if $t = 2$, $\delta \in \g^{\times}$, $\xi(\delta) = 0$, and $\nu \in 2\mathbf{Z}$,} \\ 
 & \text{or $t = 2$, $\delta \in \g^{\times}$, $\xi(\delta) = 0$, and $\nu \not\in 2\mathbf{Z}$,} \\ 
\g & \text{if $t = 2$, $\delta \in \g^{\times}$, $\xi(\delta) = -1$, and $\nu \in 2\mathbf{Z}$,} \\ 
\pe^{2} & \text{if $t = 2$, $\delta \in \g^{\times}$, $\xi(\delta) = -1$, and $\nu \not\in 2\mathbf{Z}$,} \\ 
\pe^{\kappa +1} & \text{if $t = 3$, $\delta \in \pi\g^{\times}$, and $\nu \in 2\mathbf{Z}$,} \\ 
 & \text{or $t = 3$, $\delta \in \pi\g^{\times}$, $\xi(\delta q) = 0$, and $\nu \not\in 2\mathbf{Z}$,} \\ 
 & \text{or $t = 3$, $\delta \in \pi\g^{\times}$, $\xi(\delta q) = -1$, and $\nu \not\in 2\mathbf{Z}$,} \\ 
\pe^{\kappa +2} & \text{if $t = 3$, $\delta \in \g^{\times}$, and $\nu \not\in 2\mathbf{Z}$,} \\ 
 & \text{or $t = 3$, $\delta \in \g^{\times}$, $\xi(\delta q) = 0$, and $\nu \in 2\mathbf{Z}$,} \\ 
 & \text{or $t = 3$, $\delta \in \g^{\times}$, $\xi(\delta q) = -1$, and $\nu \in 2\mathbf{Z}$,} \\ 
\pe^{2} & \text{if $t = 4$.} 
\end{cases} \label{dl}
\end{gather} 
\begin{gather} 
[\widetilde{M}/M] = 
\begin{cases}
\pe^{\kappa} & \text{if $t = 2$, $\delta \in \pi\g^{\times}$, and $\nu \not\in 2\mathbf{Z}$,} \\ 
 & \text{or $t = 2$, $\delta \in \g^{\times}$, $\xi(\delta) = 0$, and $\nu \in 2\mathbf{Z}$,} \\ 
 & \text{or $t = 2$, $\delta \in \g^{\times}$, $\xi(\delta) = -1$, and $\nu \in 2\mathbf{Z}$,} \\ 
\pe^{\kappa +1} & \text{if $t = 2$, $\delta \in \pi\g^{\times}$, and $\nu \in 2\mathbf{Z}$,} \\ 
 & \text{or $t = 2$, $\delta \in \g^{\times}$, $\xi(\delta) = 0$, and $\nu \not\in 2\mathbf{Z}$,} \\ 
 & \text{or $t = 2$, $\delta \in \g^{\times}$, $\xi(\delta) = -1$, and $\nu \not\in 2\mathbf{Z}$,} \\ 
\pe^{2\kappa +1} & \text{if $t = 3$ and $\ord(\delta q) \not\in 2\mathbf{Z}$,} \\ 
\pe^{d} & \text{if $t = 3$, $\ord(\delta q) \in 2\mathbf{Z}$, and $\xi(\delta q) = 0$,} \\ 
\g & \text{if $t = 3$, $\delta \in \pi\g^{\times}$, $\xi(\delta q) = -1$, and $\nu \not\in 2\mathbf{Z}$,} \\ 
\pe^{2} & \text{if $t = 3$, $\delta \in \g^{\times}$, $\xi(\delta q) = -1$, and $\nu \in 2\mathbf{Z}$,} \\ 
\pe^{\kappa +1} & \text{if $t = 4$ and $\nu \not\in 2\mathbf{Z}$,} \\ 
\pe^{\kappa +2} & \text{if $t = 4$ and $\nu \in 2\mathbf{Z}$.} 
\end{cases} \label{dm}
\end{gather}

\begin{cor} \label{co1}
Let the notation and assumption be the same as in \THM{t2}. 
Then the ideal $\mathfrak{p}^{\lambda(q)}$ satisfies the following properties: 
\begin{enumerate}
\item $\lambda(q)$ is determined by the isomorphism class of $(V,\, \p)$ and $q$. 
\item $\lambda(q)$ coincides with $\tau(q)$ given in \REF{re0}; 
namely, 
$\lambda(q) = \mathrm{Max} \{0 \le m \in \mathbf{Z} \mid L[q,\, 2^{-1}\mathfrak{p}^{m}] \ne \emptyset\}$. 
\item $L \cap W$ is $\g$-maximal in $W$ if and only if $2\p(h,\, L) = \mathfrak{p}^{\lambda(q)}$. 
\item $2q[\widetilde{L}/L] = \mathfrak{p}^{2\lambda(q)}[\widetilde{M}/M]$. 
In particular, 
$2q[\widetilde{L}/L][\widetilde{M}/M]^{-1}$ 
is a square ideal in $F$ contained in $\g$ whenever $L[q] \ne \emptyset$. 
\end{enumerate}
\end{cor}
We should note that assertion (3) was proven in \cite[Theorem 5.3]{Y}. 
\begin{proof}
All assertions, except (4), can easily be seen from \THM{t2} together with \cite[(6.1)]{Y}. 
Notice that we always assume $\delta \in \g^{\times} \cup \pi\g^{\times}$. 

To prove (4), 
suppose that $\p$ is isotropic. 
Then $qe_{1} + f_{1}$ belongs to $L[q,\, 2^{-1}\g]$ 
in a Witt decomposition of $(V,\, \p)$ as stated in \REF{w2}. 
Hence \THM{t2} gives $[M_{0}/L \cap W_{0}] = \mathfrak{p}^{\lambda(q)}$, 
where $W_{0}$ is the orthogonal complement of $F(qe_{1} + f_{1})$ 
and $M_{0}$ is a $\g$-maximal lattice in $W_{0}$. 
It can be seen that $L \cap W_{0} = N + \g(qe_{1} - f_{1}) + \sum_{i=2}^{r} (\g e_{i} + \g f_{i})$. 
From this we have $[(L \cap W_{0})\widetilde{\,}/L \cap W_{0}] = 2q[\widetilde{L}/L]$. 
This combined with \REF{id2} shows 
\begin{gather*}
2q[\widetilde{L}/L] = \mathfrak{p}^{2\lambda(q)}[\widetilde{M_{0}}/M_{0}]. 
\end{gather*}
Because $\p[qe_{1} + f_{1}] = \p[h]$, 
$W_{0}$ is isomorphic to $W$ under some element of $SO^{\p}(V)$ 
and so $[\widetilde{M_{0}}/M_{0}] = [\widetilde{M}/M]$. 
Thus we have the assertion when $\p$ is isotropic. 

If $\p$ is anisotropic, 
then the dimension of $V$ is $2,\ 3$, or $4$. 
As mentioned after the proof of \THM{t2}, 
$L[q] = L[q,\, 2^{-1}\pe^{\lambda(q)}]$ and $\lambda(q)$ is given by \REF{corid}. 
On the other hand, 
the ideal $2q[\widetilde{L}/L][\widetilde{M}/M]^{-1}$ can be determined by employing \REF{dl} and \REF{dm}. 
Then we find that $\pe^{2\lambda(q)} = 2q[\widetilde{L}/L][\widetilde{M}/M]^{-1}$, 
which gives the desired fact. 
\end{proof}

\section{Theorem on $[M/L \cap W]$ in the global case}

\subsection{The ideal $\mathfrak{b}(q)$}

Let $F$ be an algebraic number field and $\g$ the ring of all algebraic integers in $F$. 
For a quadratic space $(V,\, \p)$ of dimension $n\,(> 1)$ over $F$ 
and a nonzero element $q$ of $\p[V] \cap F^{\times}$, 
put $W = (Fh)^{\perp}$ and let $\psi$ be the restriction of $\p$ to $W$ 
with an arbitrary element $h$ of $V$ such that $\p[h] = q$. 
Let $[\widetilde{L}/L]$ (resp.\ $[\widetilde{M}/M]$) be 
the discriminant ideal of $(V,\, \p)$ (resp.\ $(W,\, \psi)$). 
Then we define a $\g$-ideal $\mathfrak{b}(q)$ of $F$ by 
\begin{gather}
2q[\widetilde{L}/L] = \mathfrak{b}(q)^{2}[\widetilde{M}/M]. \label{t41}
\end{gather}
This is well defined by \LEM{bq} below and 
it is determined by the isomorphism class of $(V,\, \p)$ and $q$. 
In particular, 
$\mathfrak{b}(q)$ does not depend on the choice of $h$. 
\begin{lem} \label{bq}
Let the notation be as above with $q \in \p[V] \cap F^{\times}$. 
Then $2q[\widetilde{L}/L][\widetilde{M}/M]^{-1}$ is the square of a fractional ideal in $F$. 
\end{lem}
\begin{proof}
For $q \in \p[V] \cap F^{\times}$, 
replacing $h$ by $ch$ with a suitable $0 \ne c \in \g$, 
we may assume that $\p[h] = q \in \g$ without changing $W$. 
Let $L$ be a $\g$-maximal lattice in $V$ with respect to $\p$ 
and take a Witt decomposition of $(V,\, \p)_{v}$ as in \REF{w2} for each $v \in \mathbf{h}$. 
If $\p_{v}$ is anisotropic, 
then $h \in L_{v}[q]$ because $L_{v} = \{x \in V_{v} \mid \p_{v}[x] \in \g_{v}\}$. 
Thus \COR{co1}(4) guarantees that 
$2q[\widetilde{L_{v}}/L_{v}][\widetilde{M_{v}}/M_{v}]^{-1}$ is square. 
If $\p_{v}$ is isotropic, 
we have $qe_{1} + f_{1} \in L_{v}[q]$. 
Hence by \COR{co1}(4), 
$2q[\widetilde{L_{v}}/L_{v}][\widetilde{M_{v}}/M_{v}]^{-1}$ is square, 
since the discriminant ideal of $(F_{v}(qe_{1} + f_{1}))^{\perp}$ is the same as that of $W_{v}$. 
Consequently, 
$2q[\widetilde{L_{v}}/L_{v}][\widetilde{M_{v}}/M_{v}]^{-1}$ is square for every $v \in \mathbf{h}$, 
which leads to the required fact. 
\end{proof}
\begin{thm} \label{t3}
Let $(V,\, \p)$ be a quadratic space of dimension $n\,(> 1)$ over $F$ 
and $L$ a $\g$-maximal lattice in $V$ with respect to $\p$. 
Let $h$ be an element of $V$ such that $\p[h] \in F^{\times}$. 
Put $W = (Fh)^{\perp}$ and let $\psi$ be the restriction of $\p$ to $W$. 
Then the following assertions hold: 
\begin{gather}
\p(h,\, L)[M/L \cap W] = 2^{-1}\mathfrak{b}(q), \label{t31} \\ 
c\mathfrak{b}(q) = \mathfrak{b}(c^{2}q) = \prod_{v \in \mathbf{h}} \pe_{v}^{\lambda_{v}(c^{2}q)}, \qquad 
\mathfrak{b}(q) \subset 2\p(h,\, L) \subset c^{-1}\g. \label{t32} 
\end{gather}
Here $M$ is a $\g$-maximal lattice in $W$ with respect to $\psi$, 
$\mathfrak{b}(q)$ is the ideal defined by \REF{t41} with $q = \p[h]$, 
$c$ is a nonzero element of $\g$ such that $ch \in L$, 
$\pe_{v}$ is the prime ideal of $F$ corresponding to $v$, 
and $\lambda_{v}(c^{2}q)$ is the integer given in \REF{id1} with $c^{2}q$ in place of $q$. 
In particular, 
if $h \in L$, 
then 
\begin{gather}
\mathfrak{b}(q) = \prod_{v \in \mathbf{h}} \pe_{v}^{\lambda_{v}(q)}, \qquad 
\mathfrak{b}(q) \subset 2\p(h,\, L) \subset \g. \label{t33}
\end{gather}
\end{thm}
\begin{proof}
We first prove the case $h \in L$. 
Since \THM{t2} is applicable, 
by localization we obtain 
\begin{gather*}
[M/L \cap W] = \left\{ \prod_{v \in \mathbf{h}} \pe_{v}^{\lambda_{v}(q)} \right\}\cdot (2\p(h,\, L))^{-1}. 
\end{gather*}
Furthermore, 
\COR{co1}(4) shows that $\mathfrak{b}(q)_{v} = \pe_{v}^{\lambda_{v}(q)}$ for every $v \in \mathbf{h}$. 
Thus we have $[M/L \cap W] = \mathfrak{b}(q)(2\p(h,\, L))^{-1}$. 
In view of \LEM{lid}(6), 
this implies $\mathfrak{b}(q) \subset 2\p(h,\, L)$. 
Since $2\p(h,\, L) \subset \g$, 
we have \REF{t33}. 

As for the general case, 
we can always take $0 \ne c \in \g$ so that $ch \in L$. 
Since $ch \in L[c^{2}q]$, 
\REF{t31} is applicable to $ch$. 
Noticing $W = (Fch)^{\perp}$, 
we have then 
$[M/L \cap W] = \mathfrak{b}(c^{2}q)(2\p(ch,\, L))^{-1}$. 
The ideal $\mathfrak{b}(c^{2}q)$ is determined by 
\begin{gather*}
2c^{2}q[\widetilde{L}/L] = \mathfrak{b}(c^{2}q)^{2}[\widetilde{M}/M]. 
\end{gather*}
On the other hand, 
since $q \in \p[V] \cap F^{\times}$, 
$\mathfrak{b}(q)$ is meaningful and it satisfies \REF{t41}. 
From these it follows that $\mathfrak{b}(c^{2}q) = c\mathfrak{b}(q)$. 
Thus we obtain 
\begin{gather*}
[M/L \cap W] = c\mathfrak{b}(q)(2\p(ch,\, L))^{-1} = \mathfrak{b}(q)(2\p(h,\, L))^{-1}, 
\end{gather*}
which proves \REF{t31} in the general case. 
Moreover these results combined with \COR{co1}(4) prove \REF{t32}. 
This completes the proof. 
\end{proof}
\begin{cor} \label{co2}
Let the notation and assumption be the same as in \THM{t3}. 
Then $L \cap W$ is $\g$-maximal in $W$ if and only if 
$\p[h](\p(h,\, L))^{-2} = 2[\widetilde{M}/M][\widetilde{L}/L]^{-1}$. 
\end{cor}

\subsection{On the class number of $O^{\psi}(W)$}

Let $(V,\, \p)$ be a quadratic space of dimension $n$ over a number field $F$ 
and $L$ a $\g$-lattice in $V$. 
Let $O^{\p}(V)_{\mathbf{A}}$ and $SO^{\p}(V)_{\mathbf{A}}$ be the adelization of $O^{\p}(V)$ and $SO^{\p}(V)$, 
respectively (cf.\ \cite[{\S}9.6]{04}). 
For $\alpha \in O^{\p}(V)_{\mathbf{A}}$ 
we denote by $L\alpha$ the $\g$-lattice in $V$ whose localization 
is given by $L_{v}\alpha_{v}$ for each $v \in \mathbf{h}$. 
We put 
\begin{gather*}
D(L) = \{ \alpha \in O^{\p}(V)_{\mathbf{A}} \mid L\alpha = L \}, \quad 
C(L) = SO^{\p}(V)_{\mathbf{A}} \cap D(L), \\ 
\Gamma(L) = O^{\p}(V) \cap D(L). 
\end{gather*}
Then the map $L\alpha^{-1} \mapsto \alpha$ gives a bijection of 
$\{L\alpha\mid\alpha\in O^{\p}_{\mathbf{A}}\}/O^{\p}$ onto $O^{\p}\setminus O^{\p}_{\mathbf{A}}/D(L)$. 
We call $\{L\alpha\mid\alpha\in O^{\p}_{\mathbf{A}}\}$ the $O^{\p}$-\textit{genus} of $L$, 
$\{L\gamma\mid\gamma\in O^{\p}\}$ the $O^{\p}$-\textit{class} of $L$, 
and $\#[\{L\alpha\mid\alpha\in O^{\p}_{\mathbf{A}}\}/O^{\p}]$ 
the \textit{class number} of $O^{\p}$ relative to $D(L)$ 
or the class number of the genus of $L$ with respect to $O^{\p}$. 
These terms are also defined for $SO^{\p}$ in a similar manner. 
It is known that 
all $\g$-maximal lattices in $V$ with respect to $\p$ form a single $O^{\p}$-genus 
and that coincides with the $SO^{\p}$-genus. 
In particular when $n$ is odd, 
the class number of $O^{\p}$ relative to $D(L)$ equals 
the class number of $SO^{\p}$ relative to $C(L)$; 
see \cite[Lemma 9.23(i)\ and\ (ii)]{04}, for example. 
\\

Let $h$ be an element of $V$ such that $\p[h] \ne 0$. 
Put $W = (Fh)^{\perp}$ and let $\psi$ be the restriction of $\p$ to $W$. 
We then identify $O^{\psi}(W)$ with the subgroup 
$\{\gamma \in O^{\p}(V) \mid h\gamma = h\}$ of $O^{\p}(V)$. 
Clearly $D(L \cap W)$ contains $O^{\psi}(W)_{\mathbf{A}} \cap D(L)$. 

Now we assume that $L$ is $\g$-maximal with respect to $\p$ and $n > 2$. 
Then the formula in \cite[Theorem 11.6(iii)]{04} due to Shimura shows that 
\begin{gather}
\# \left\{O^{\psi}\setminus O^{\psi}_{\mathbf{A}}/(O^{\psi}_{\mathbf{A}} \cap D(L)) \right\} 
	= \sum_{\alpha} \# \left\{L\alpha^{-1}[q,\, \mathfrak{b}]/\Gamma(L\alpha^{-1})\right\}, 
\label{116}
\end{gather}
where $q = \p[h]$, $\mathfrak{b} = \p(h,\, L)$, 
and $\alpha$ runs over all representatives for $O^{\p}\setminus O^{\p}_{\mathbf{A}}/D(L)$. 
Therefore, 
if we find that 
$\#\{O^{\psi}\setminus O^{\psi}_{\mathbf{A}}/D(L \cap W)\}$ 
coincides with the number of the left-hand side of \REF{116}, 
then the class number of the genus of $L \cap W$ can be obtained from the right-hand side of \REF{116}. 
In fact, 
we can derive such a coincidence by means of \cite[Proposition 11.13]{04} 
for odd dimensional spaces satisfying some conditions on $h$. 

Our purpose in this subsection is to prove the following proposition: 
\begin{prop} \label{cnf}
Let $(V,\, \p)$ be a even dimensional quadratic space over a number field $F$. 
Assume that the discriminant field is $F$. 
Let $L$ be a $\g$-maximal lattice in $V$ with respect to $\p$ 
and $h$ an element of $V$ such that $\p[h] = q \ne 0$. 
Put $W = (Fh)^{\perp}$ and let $\psi$ be the restriction of $\p$ to $W$. 
Identify $O^{\psi}(W)$ with $\{\gamma \in O^{\p}(V) \mid h\gamma = h\}$. 
Assume that $q(\p(h,\, L))^{-2} = 2[\widetilde{M}/M][\widetilde{L}/L]^{-1}$ and 
$\# \{v \in \mathbf{h} \mid q(2\p(h,\, L)_{v})^{-2} = \pe_{v}\} \le 1$, 
where $M$ is a $\g$-maximal lattice in $W$ with respect to $\psi$. 
Then the following assertions hold: 
\begin{enumerate}
\item The discriminant field of $(W,\, \psi)$ is $F(\sqrt{-q})$ and 
	the characteristic algebra $Q(\psi)_{v}$ coincides with $Q(\p)_{v}$ for every $v \in \mathbf{h}$. 
\item $L \cap W$ is $\g$-maximal in $W$ with respect to $\psi$. 
\item $[D(L \cap W):O^{\psi}_{\mathbf{A}} \cap D(L)] = [C(L \cap W):SO^{\psi}_{\mathbf{A}} \cap C(L)] = 1$ 
	or $2$ according as 
	$\# \{v \in \mathbf{h} \mid q(2\p(h,\, L)_{v})^{-2} = \pe_{v}\} = 0$ or $1$. 
\item $O^{\psi}\varepsilon D(L \cap W) = O^{\psi}\varepsilon (O^{\psi}_{\mathbf{A}} \cap D(L))$ 
	for every $\varepsilon \in O^{\psi}_{\mathbf{A}}$. 
\end{enumerate}
\end{prop}
Let us here insert some notation and facts. 
Given $(V,\, \p)$ of dimension $n$ over a number field $F$, 
let $G(V)$ and $G^{+}(V)$ be the Clifford group of $\p$ and the even Clifford group of $\p$, 
respectively. 
We define a homomorphism $\tau$ as follows: 
\begin{gather}
\tau : G(V) \longrightarrow O^{\p}(V) \quad 
	\text{via}\quad x\tau(\alpha) = \alpha^{-1}x\alpha \quad \text{for $x \in V$}. 
\label{tau}
\end{gather}
This is surjective when $n$ is even and $\tau(G(V)) = SO^{\p}(V)$ when $n$ is odd. 
Moreover $\tau$ gives an isomorphism of $G^{+}(V)/F^{\times}$ onto $SO^{\p}(V)$ (cf.\ \cite[Theorem 3.6]{04}). 

For a quadratic space $(V,\, \p)_{v}$ over a local field $F_{v}$ with $v \in \mathbf{h}$, 
we denote by $A(V)_{v}$ and $G(V)_{v}$ 
the Clifford algebra of $\p_{v}$ and the Clifford group of $\p_{v}$, respectively. 
We can also define a homomorphism of $G(V)_{v}$ into $O^{\p}(V)_{v}$ in the same manner as in \REF{tau}; 
we denote it by the same symbol $\tau$. 
For a $\g_{v}$-maximal lattice $L_{v}$ in $V_{v}$ with respect to $\p_{v}$, 
we define a subgroup $J(V)_{v}$ of $G^{+}(V)_{v}$ by 
\begin{gather}
J(V)_{v} = \{\alpha \in G^{+}(V)_{v} \mid 
	\tau(\alpha) \in C(L_{v}),\ \alpha \alpha^{*} \in \g_{v}^{\times}\}, \label{j} 
\end{gather}
where $*$ is the canonical involution of $A(V)_{v}$. 
Let $A(L_{v})$ be the subring of $A(V)_{v}$ 
generated by $\g_{v}$ and $L_{v}$. 
Also put $A^{+}(L_{v}) = A(L_{v}) \cap A^{+}(V)_{v}$. 
Then $A(L_{v})$ (resp.\ $A^{+}(L_{v})$) is an order in 
$A(V)_{v}$ (resp.\ $A^{+}(V)_{v}$) 
(cf.\ \cite[{\S}8.2]{04}). 
\\

\begin{proof}
Since the discriminant field of $\p$ is $F$, 
we have $t_{v} = 0$ or $4$ for every $v \in \mathbf{h}$ by \REF{cd1}
and $[\widetilde{L}/L] = D_{Q(\p)}^{2}$ by \THM{di}. 
The core dimension of $\psi$ at $v$ is $1$ or $3$ and 
$Q(\psi)_{v}$ is $M_{2}(F_{v})$ or a division algebra $Q(\p)_{v}$ 
according as $t_{v} = 0$ or $4$ by \THM{t1} and \REF{cd2}. 
Thus $Q(\psi)_{v} = Q(\p)_{v}$ for every $v \in \mathbf{h}$ 
and the same theorem gives the invariants of $\psi$, 
which proves (1). 

The first assumption on $h$ implies that $2\p(h,\, L) = \mathfrak{b}(q)$, 
and hence by \THM{t3}, 
$L \cap W$ is $\g$-maximal in $W$. 
This proves (2). 

Let $v \in \mathbf{h}$ and suppose $t_{v} = 4$. 
Then there exists an element $z$ of a core subspace $Z_{v}$ of $(V,\, \p)_{v}$ 
as in \REF{w2} such that $\p_{v}[z] = q$. 
Then $2\p_{v}(z,\, N_{v}) = \mathfrak{b}(q)_{v}$, 
where $N_{v} (= L_{v} \cap Z_{v})$ is the $\g_{v}$-maximal lattice in $Z_{v}$. 
Indeed, 
$cz \in N_{v}[c^{2}q] = N_{v}[c^{2}q,\, 2^{-1}\mathfrak{b}(c^{2}q)_{v}]$ 
with a suitable element $c$ of $F^{\times}$ as observed after \REF{corid}; 
this combined with $\mathfrak{b}(c^{2}q)_{v} = c\mathfrak{b}(q)_{v}$ 
gives $\p_{v}(z,\, N_{v}) = 2^{-1}\mathfrak{b}(q)_{v}$. 
Since $\widetilde{L}_{v} \ne L_{v}$ and 
$2\p_{v}(z,\, L_{v}) = \mathfrak{b}(q)_{v} = 2\p(h,\, L)_{v}$, 
\cite[Proposition 11.12(iv)\ and\ (v)]{04} are applicable to $(V,\, \p)_{v}$, $L_{v}$, and $h$; 
thus $[C(L \cap W)_{v}:SO^{\psi}_{v} \cap C(L)_{v}] = 1$ and 
there is an element $\gamma \in O^{\psi}_{v} \cap D(L)_{v}$ such that $\det(\gamma) = -1$. 
From these we have $[D(L \cap W)_{v}:O^{\psi}_{v} \cap D(L)_{v}] = 1$. 
We note that $q\p(2h,\, L)_{v}^{-2} = \g_{v}$ or $\pe_{v}^{-1}$ 
according as $\mathrm{ord}_{v}(q)$ is even or odd. 

Suppose $t_{v} = 0$. 
Then $[\widetilde{L}/L]_{v} = \g_{v}$. 
By \THM{di} we have $[\widetilde{M}/M]_{v} = 2\g_{v}$ or $2\pe_{v}$ 
and hence $\mathfrak{b}(q)_{v}^{2} = q\g_{v}$ or $q\pe_{v}^{-1}$ 
according as $\mathrm{ord}_{v}(q)$ is even or odd. 

If $\mathrm{ord}_{v}(q)$ is even, 
then $q\p(2h,\, L)_{v}^{-2} = q\mathfrak{b}(q)_{v}^{-2} = \g_{v}$, 
so that \cite[Proposition 11.12(iv) and (v)]{04} are applicable; 
hence $[C(L \cap W)_{v}:SO^{\psi}_{v} \cap C(L)_{v}] = 1$ and 
$O^{\psi}_{v} \cap D(L)_{v}$ has an element $\gamma$ such that $\det(\gamma) = -1$. 
Then $[D(L \cap W)_{v}:O^{\psi}_{v} \cap D(L)_{v}] = 1$. 

If $\mathrm{ord}_{v}(q)$ is odd, 
then $q\p(2h,\, L)_{v}^{-2} = \pe_{v}$. 
This happens for at most one prime by our assumption. 
When such a prime does not exist, 
we have $D(L \cap W) = O^{\psi}_{\mathbf{A}} \cap D(L)$ and hence assertions (3) and (4) are proved. 

Hereafter until the end of the proof, 
we assume that there exists a prime $u \in \mathbf{h}$ such that $q\p(2h,\, L)_{u}^{-2} = \pe_{u}$. 
In the same way as in the proof of \cite[Proposition 11.12(iv)]{04}, 
we can find an element $\gamma \in O^{\psi}_{u}$ such that $\det(\gamma) = -1$ and $L_{u}\gamma = L_{u}$; 
see the case where $\widetilde{L} = L$ and $t \ne 1$ in that proof. 
Thus we have 
$[D(L \cap W)_{u}:O^{\psi}_{u} \cap D(L)_{u}] = [C(L \cap W)_{u}:SO^{\psi}_{u} \cap C(L)_{u}]$. 
Now, 
$(L \cap W)_{u}$ is $\g_{u}$-maximal in $W_{u}$. 
We then consider subgroups $J(V)_{u}$ and $J(W)_{u}$ defined by 
\REF{j} with $L_{u}$ and $(L \cap W)_{u}$, respectively. 
It is noted by \cite[Lemma 3.16]{04} that $G^{+}(W)_{u}$ coincides with 
the subgroup $\{\alpha \in G^{+}(V)_{u} \mid \alpha h = h\alpha\}$. 
Applying \cite[Theorem 8.9(i)]{04} to $(V,\, \p)_{u}$ and $L_{u}$, 
we have $\tau(J(V)_{u}) = C(L)_{u}$. 
Moreover, 
noticing that $\mathrm{ord}_{u}(q)$ is odd and 
applying \cite[Theorem 1.8(iii)]{06a} to $(W,\, \psi)_{u}$ and $(L \cap W)_{u}$, 
we have $[C(L \cap W)_{u}:\tau(J(W)_{u})] = 2$. 
To observe $J(W)_{u}$, 
we further consider the order $A(L)_{u}$ (resp.\ $A(L \cap W)_{u}$) 
in $A(V)_{u}$ (resp.\ $A(W)_{u}$) defined before the proof. 
Then $A(L)_{u}$ is a maximal order containing $J(V)_{u}$ and $L_{u}$ 
by \cite[Theorem 8.6(i),\, (v),\, and\ Lemma 8.4(ii)]{04}. 
Hence $J(V)_{u} = G^{+}(V)_{u} \cap A(L)_{u}^{\times}$ by \cite[Proposition 8.8(i)]{04}. 
Similarly, 
$A(L \cap W)_{u}$ is a maximal order containing $J(W)_{u}$ and $(L \cap W)_{u}$. 
This can be proven in a similar way to the proof of \cite[Theorem 8.6(i)\ and\ (v)]{04} 
with a unique maximal order in the Clifford algebra of a $1$-dimensional core subspace of $W_{u}$. 
Thus by \cite[Proposition 8.8(i)]{04} we have 
$J(W)_{u} = G^{+}(W)_{u} \cap A(L \cap W)_{u}^{\times}$. 
Therefore $J(W)_{u}$ is contained in $G^{+}(W)_{u} \cap J(V)_{u}$. 
Since the converse inclusion is clear, 
we obtain $J(W)_{u} = G^{+}(W)_{u} \cap J(V)_{u}$. 
This combined with the above facts shows $[C(L \cap W)_{u}:SO^{\psi}_{u} \cap C(L)_{u}] = 2$. 
Summing up all indices observed above, 
we have $[D(L \cap W):O^{\psi}_{\mathbf{A}} \cap D(L)] = [C(L \cap W):SO^{\psi}_{\mathbf{A}} \cap C(L)] = 2$, 
which proves (3). 

To prove (4), 
we borrow the idea of the proof of \cite[Proposition 11.13(ii)]{04}. 
Put $G = O^{\p}(V)$, $H = O^{\psi}(W)$, and $\mathfrak{a} = \p[h](2\p(h,\, L))^{-2}$. 
Fix an arbitrary element $\varepsilon$ of $H_{\mathbf{A}}$. 
We put $\Lambda = L\varepsilon^{-1}$, 
which is a $\g$-maximal lattice in $(V,\, \p)$. 
Then $D(\Lambda \cap W) = \varepsilon D(L \cap W)\varepsilon^{-1}$ and 
$H_{\mathbf{A}} \cap D(\Lambda) = \varepsilon (H_{\mathbf{A}} \cap D(L))\varepsilon^{-1}$. 
Since $[D(L \cap W):H_{\mathbf{A}} \cap D(L)] = 2$ by (3), 
we have $[D(\Lambda \cap W):H_{\mathbf{A}} \cap D(\Lambda)] = 2$. 
We are going to show that 
$\tau(h)$ is an element of $H$ such that 
$(\Lambda \cap W)\tau(h) = \Lambda \cap W$ and $\Lambda \tau(h) \ne \Lambda$. 
Once that is proven, 
we obtain 
\begin{gather*}
\varepsilon D(L \cap W)\varepsilon^{-1} = 
	\varepsilon (H_{\mathbf{A}} \cap D(L))\varepsilon^{-1} \sqcup 
	\tau(h) \varepsilon (H_{\mathbf{A}} \cap D(L))\varepsilon^{-1}. 
\end{gather*}
Thus $H\varepsilon D(L \cap W)\varepsilon^{-1} = H\varepsilon (H_{\mathbf{A}} \cap D(L))\varepsilon^{-1}$ 
because $\tau(h) \in H$, 
which leads to the desired conclusion. 

First, 
$\tau(h) \in G$ and $h\tau(h) = h$, 
and hence $\tau(h)$ belongs to $H$. 
Since $\varepsilon_{v} \in H_{v}$ for every $v \in \mathbf{h}$, 
we have $\p(h,\, \Lambda) = \p(h,\, L)$. 
Taking $c \in F_{\mathbf{A}}^{\times}$ so that $2c\p(h,\, L) = \g$, 
we see that $2\p(c_{v}h,\, \Lambda_{v}) = \g_{v}$ and 
$\p[c_{v}h]\g_{v} = \p[c_{v}h]\p(2c_{v}h,\, \Lambda_{v})^{-2} = \mathfrak{a}_{v}$ 
for every $v \in \mathbf{h}$. 
In the proof of (3) we have seen that $\mathfrak{a}_{v} = \g_{v},\, \pe_{v}^{-1}$, or $\pe_{v}$; 
the last case happens only at $v = u$ denoted there. 
Suppose $\mathfrak{a}_{v} = \g_{v}$. 
Then $\p[c_{v}h]\g_{v} = \g_{v}$ and $2\p(c_{v}h,\, \Lambda_{v}) = \g_{v}$. 
Hence $c_{v}h$ belongs to $\Lambda_{v}$ and it is invertible in the order $A(\Lambda)_{v}$. 
Since this order contains $\Lambda_{v}$ by definition, 
$A(\Lambda)_{v} \cap V_{v} = \Lambda_{v}$ by \cite[Lemma 8.4(iii)]{04}. 
Thus we have $\Lambda_{v}\tau(h) = h^{-1}A(\Lambda)_{v}h \cap V_{v} = \Lambda_{v}$. 
Suppose $\mathfrak{a}_{v} \ne \g_{v}$. 
We need a Witt decomposition as in \REF{w2}: 
\begin{gather*}
V_{v} = Z_{v} \oplus \sum_{i=1}^{r}(F_{v}e_{i} + F_{v}f_{i}),\quad 
\Lambda_{v} = N_{v} + \sum_{i=1}^{r}(\g_{v}e_{i} + \g_{v}f_{i}), 
\end{gather*}
where $N_{v}$ is a unique $\g_{v}$-maximal lattice in a core subspace $Z_{v}$. 
If $t_{v} = 4$, 
then $\mathfrak{a}_{v} = \pe_{v}^{-1}$. 
We can find an element $z$ of $Z_{v}$ such that 
$\p[z] = \p[c_{v}h] \in \pi_{v}^{-1}\g_{v}^{\times}$ and 
$2\p(z,\, N_{v}) = \g_{v} = 2\p(c_{v}h,\, \Lambda_{v})$. 
This is because the core space is isomorphic to $(Q(\p)_{v},\, \beta_{v})$ 
and $2\p(z,\, N_{v}) = \mathrm{Tr}_{Q(\p)_{v}/F_{v}}(z\mathfrak{o}_{v}) = \g_{v}$, 
where $\mathfrak{o}_{v}$ is a unique maximal order in the division algebra $Q(\p)_{v}$. 
By virtue of \cite[Theorem 1.3]{06a}, 
there exists $\alpha \in G_{v}$ such that 
$c_{v}h = z\alpha$ and $\Lambda_{v}\alpha = \Lambda_{v}$. 
Moreover by \cite[Lemma 3.8(ii)]{04}, 
$\tau(z\alpha) = \alpha^{-1}\tau(z)\alpha$. 
Thus noticing $\tau(z) \in O^{\p}(Z_{v}) = D(N_{v})$, 
we have $\Lambda_{v}\tau(h) = \Lambda_{v}\alpha^{-1}\tau(z)\alpha = \Lambda_{v}$. 
If $t_{v} = 0$, 
then $\mathfrak{a}_{v} = \pe_{v}$ and $v = u$. 
Put $p = \p[c_{u}h] \in \pe_{u}$ and $k = pe_{1} + f_{1}$. 
It can be seen that $\p[k] = \p[c_{u}h]$ and 
$\p(k,\, \Lambda_{u}) = 2^{-1}\g_{u} = \p(c_{u}h,\, \Lambda_{u})$. 
Again by \cite[Theorem 1.3]{06a}, 
$c_{u}h\gamma = k$ and $\Lambda_{u}\gamma = \Lambda_{u}$ with some $\gamma \in G_{u}$. 
Moreover $\tau(h\gamma) = \gamma^{-1}\tau(h)\gamma$ by \cite[Lemma 3.8(ii)]{04}. 
Then $\gamma$ gives an isomorphism of $W_{u}$ onto $W_{u}^{\prime} = (F_{u}k)^{\perp}$ 
such that $(\Lambda_{u} \cap W_{u})\gamma = \Lambda_{u} \cap W_{u}^{\prime}$ 
and $\gamma^{-1}\tau(h)\gamma = \tau(k)$. 
We see that 
$\Lambda_{u} \cap W_{u}^{\prime} = 
	N_{u} + \g_{u}(pe_{1}-f_{1}) + \sum_{i=2}^{r}(\g_{u}e_{i} + \g_{u}f_{i})$. 
Then employing \cite[Lemma 3.10]{04}, 
we can find that 
\begin{gather*}
(\Lambda_{u} \cap W_{u}^{\prime})\tau(k) 
	= \{-x - a(pe_{1} - f_{1}) \mid x \in \Lambda_{u} \cap U,\ a \in \g_{u}\} 
	= \Lambda_{u} \cap W_{u}^{\prime}, \\ 
\Lambda_{u}\tau(k) 
	= \{-x + pae_{1} + p^{-1}bf_{1} \mid x \in \Lambda_{u} \cap U,\ a,\, b \in \g_{u}\} 
	\ne \Lambda_{u}, 
\end{gather*}
where $U = Z_{u} + \sum_{i=2}^{r}(F_{u}e_{i} + F_{u}f_{i})$. 
Thus we obtain 
$(\Lambda_{u} \cap W_{u})\tau(h) = \Lambda_{u} \cap W_{u}$ 
and $\Lambda_{u}\tau(h) \ne \Lambda_{u}$. 
Consequently, 
$\tau(h) \in H$, 
$(\Lambda \cap W)\tau(h) = \Lambda \cap W$, 
and $\Lambda \tau(h) \ne \Lambda$, 
which are the desired facts. 
This completes the proof. 
\end{proof}

\section{Applications}

\subsection{Eight-dimensional case}

Let $(V,\, \p)$ be a quadratic space over $\mathbf{Q}$ with invariants 
\begin{gather}
\{8,\ \mathbf{Q},\ M_{2}(\mathbf{Q}),\ 8\}. \label{8}
\end{gather}
Clearly the core dimension is $0$ at any $v \in \mathbf{h}$. 
Let $q$ be a squarefree positive integer and $h$ an element of $V$ such that $\p[h] = q$. 
Put $W = (Fh)^{\perp}$ and $\psi = \p|_{W}$. 
Then the invarinats of $(W,\, \psi)$ are given by 
\begin{gather}
\{7,\ \mathbf{Q}(\sqrt{-q}),\ M_{2}(\mathbf{Q}),\ 7\}. \label{7}
\end{gather}

Let $L$ (resp. $M$) be a $\mathbf{Z}$-maximal lattice in $(V,\, \p)$ (resp. $(W,\, \psi)$); 
the discriminant ideals are given by 
$[\widetilde{L}/L] = \mathbf{Z}$ and $[\widetilde{M}/M] = 2q\mathbf{Z}$. 
Hence $\mathfrak{b}(q)=\mathbf{Z}$, 
and so $L[q] = L[q,\, 2^{-1}\mathbf{Z}]$ by \REF{t33}. 
\begin{prop}
The notation being as above, 
suppose $h \in L[q]$. 
Then the following assertions hold: 
\begin{enumerate}
\item $(V,\, \p) \cong (B,\, \beta) \oplus (B,\, \beta) \cong (\q^{1}_{8},\, 1_{8})$, 
	where $B$ is a definite quaternion algebra over $\q$ ramified exactly at $2$ with norm form $\beta$. 
\item $L[q] = L[q,\, 2^{-1}\mathbf{Z}] \ne \emptyset$ for every squarefree positive integer $q$. 
\item $L \cap W$ is $\mathbf{Z}$-maximal in $W$ with respect to $\psi$. 
\item The class number of the genus of $\mathbf{Z}$-maximal lattices in $(W,\, \psi)$ is given by 
	$\#\{L[q]/\Gamma(L)\}$ for every prime number $q$. 
\end{enumerate}
\end{prop}
\begin{proof}
It is well known that the norm form $\beta$ of $B$ is represented by 
the identity element $1_{4}$ of $\q_{4}^{4}$ with respect to a suitable basis over $\q$. 
This means that $(\q^{1}_{4},\, 1_{4}) \cong (B,\, \beta)$. 
From the argument in \cite[{\S}7.4]{04} we see that $A(1_{4}) = M_{2}(B)$. 
Viewing $1_{8}$ as $1_{4} \oplus 1_{4}$, 
by \cite[Lemma 2.8(i)]{cq} we have $A(1_{8}) \cong A(1_{4}) \otimes_{\mathbf{Q}} A(1_{4}) \cong M_{16}(\q)$. 
Hence $Q(1_{8}) = M_{2}(\q) = Q(\p)$. 
Since the other invariants of $1_{8}$ are also the same as those of $\p$, 
we have $(\q^{1}_{8},\, 1_{8}) \cong (V,\, \p)$, 
which proves (1). 
Now, 
it is known that the genus of $\mathbf{Z}$-maximal lattices 
in $(\q^{1}_{8},\, 1_{8})$ consists of a single $SO^{1_{8}}$-class. 
Since $1_{8}$ is positive definite and $0 < q \in \mathbf{Z}$, 
in view of \cite[Proposition 1.8]{99b}, 
we have $L[q] \ne \emptyset$. 
This together with $L[q] = L[q,\, 2^{-1}\mathbf{Z}]$ shows (2). 
Because $2\p(h,\, L) = \mathfrak{b}(q) = \mathbf{Z}$, 
by \REF{t31} we have (3). 
Let $q$ be a prime number. 
Then \PROP{cnf} is applicable to the present space. 
Thus (4) follows from \REF{116}. 
\end{proof}

\subsection{Six-dimensional case}

Let $(V,\, \p)$ be a quadratic space over $\mathbf{Q}$ with invariants 
\begin{gather}
\{6,\ \mathbf{Q}(\sqrt{-1}),\ B_{2,\, \infty},\ 6\}. \label{inv}
\end{gather}
Here we denote by $B_{p,\, \infty}$ a definite quaternion algebra over $\mathbf{Q}$ 
ramified exactly at a prime number $p$. 
Put $B = B_{2,\, \infty}$, $K = \mathbf{Q}(\sqrt{-1})$, and $\xi_{p} = \xi_{p}(-1)$. 
Then $D_{K/\mathbf{Q}} = 4\mathbf{Z}$, 
and $\xi_{p} = 1$ if and only if $p \equiv 1 \pmod{4}$. 
By \REF{cd1} the core dimension $t_{p}$ is given by 
\begin{gather*} 
t_{p} = 
\begin{cases}
0 & \text{if $\xi_{p} = 1$}, \\
2 & \text{otherwise}. 
\end{cases}
\end{gather*}
Hence if $\xi_{p} = 1$, 
then $(V,\, \p)_{p} = (V,\, \p) \otimes_{\mathbf{Q}} \mathbf{Q}_{p}$ 
has a split Witt decomposition. 
If $\xi_{p} \ne 1$, 
a core subspace of $(V,\, \p)_{p}$ can be identified with 
$(K_{p},\, \kappa_{p})$ or $(K_{p},\, -\kappa_{p})$ 
according as $\xi_{p} = -1$ or $\xi_{p} = 0$. 
This is because the characteristic algebra at $p$ is $M_{2}({\mathbf{Q}_{p}})$ if $p \equiv 3 \pmod{4}$, 
and it is a division algebra $B_{2} = \{K_{2},\, -1\}$ over $\mathbf{Q}_{2}$ if $p = 2$. 
Here $\kappa_{p}$ is the norm form of $K_{p} = \mathbf{Q}_{p}(\sqrt{-1})$ 
and note that $-1 \not\in \kappa_{2}[K_{2}^{\times}]$. 

Let $q$ be a prime number and $h$ an element of $V$ such that $\p[h] = q$. 
Put $W = (Fh)^{\perp}$ and $\psi = \p|_{W}$. 
Then the invarinats of $(W,\, \psi)$ are given by 
\begin{gather}
\{5,\ \mathbf{Q}(\sqrt{q}),\ B_{q,\, \infty},\ 5\} \quad \text{if $q \equiv 3 \pmod{4}$}, \label{a1} \\
\{5,\ \mathbf{Q}(\sqrt{q}),\ B_{2,\, \infty},\ 5\} \quad \text{otherwise}. \label{a2}
\end{gather}
To see this, 
let us determine $Q(\psi)_{p}$ for every prime number $p$ by using \THM{t1}. 
If $\xi_{p} = 1$, 
then $p \equiv 1 \pmod{4}$ and so $p$ is not ramified in $B$; 
hence $Q(\psi)_{p} = M_{2}(\mathbf{Q}_{p})$. 
If $\xi_{p} = -1$, 
then $p \equiv 3 \pmod{4}$, 
which is unramified in $B$. 
Since $K_{p}$ is unramified over $\mathbf{Q}_{p}$, 
$q \in \kappa_{p}[K_{p}^{\times}]$ if $p \ne q$. 
Then $Q(\psi)_{p} = M_{2}(\mathbf{Q}_{p})$ if $p \ne q$, 
and $Q(\psi)_{p}$ is a division algebra if $p = q$. 
If $\xi_{p} = 0$, 
then $p = 2$, 
which is ramified in $B$. 
Since $1 + 4\mathbf{Z}_{2} \subset \kappa_{2}[K_{2}^{\times}]$, 
it can be seen that $q \not\in \kappa_{2}[K_{2}^{\times}]$ if $q \equiv 3 \pmod{4}$, 
and $q \in \kappa_{2}[K_{2}^{\times}]$ if $q \equiv 1 \pmod{4}$ or $q = 2$. 
Hence we have $Q(\psi)_{2} = M_{2}(\mathbf{Q}_{2})$ if $q \equiv 3 \pmod{4}$, 
and $Q(\psi)_{2} = B_{2}$ otherwise $q$. 
All these $Q(\psi)_{p}$ together with $Q(\psi)_{\infty} = \mathbf{H}$ 
determine $Q(\psi)$ given in \REF{a1} and \REF{a2}. 
\\

Let $L$ (resp.\ $M$) be a $\mathbf{Z}$-maximal lattice in $V$ (resp.\ $W$) 
with respect to $\p$ (resp.\ $\psi$). 
The discriminant ideals of $\p$ and $\psi$ are given by 
\begin{gather}
[\widetilde{L}/L] = 4\mathbf{Z}, \qquad [\widetilde{M}/M] = 
\begin{cases}
8q\mathbf{Z} & \text{if $q \equiv 1 \pmod{4}$}, \\ 
2q\mathbf{Z} & \text{if $q \equiv 3 \pmod{4}$ or $q = 2$} 
\end{cases} \label{discr}
\end{gather}
by applying \THM{di} to \REF{inv}, \REF{a1}, and \REF{a2}. 
Then $\mathfrak{b}(q)$ of \REF{t41} is determined as follows: 
If $q \equiv 1 \pmod{4}$, 
then $2q[\widetilde{L}/L]$ $[\widetilde{M}/M]^{-1} = \mathbf{Z}$; 
if $q \equiv 3 \pmod{4}$ or $q = 2$, 
then $2q[\widetilde{L}/L][\widetilde{M}/M]^{-1} = 4\mathbf{Z}$ by \REF{discr}. 
Thus 
\begin{gather}
\mathfrak{b}(q) = 
\begin{cases}
\mathbf{Z} & \text{if $q \equiv 1 \pmod{4}$}, \\
2\mathbf{Z} & \text{if $q \equiv 3 \pmod{4}$ or $q = 2$} 
\end{cases} \label{b}
\end{gather}
for a prime number $q$. 
Moreover, 
by \REF{t33}, 
\begin{gather*}
L[q] = L[q,\, \mathbf{Z}] \cup L[q,\, 2^{-1}\mathbf{Z}] 
\end{gather*}
and $L[q,\, \mathbf{Z}] = \emptyset$ if $q \equiv 1 \pmod{4}$. 
With $\mathfrak{b}(q)$ in \REF{b} the ideal $[M/L \cap W]$ can be computed: 
If $h \in L[q,\, 2^{-1}\mathbf{Z}]$, 
then $[M/L \cap W]=\mathfrak{b}(q)$. 
If $h \in L[q,\, \mathbf{Z}]$, 
then $q \not\equiv 1 \pmod{4}$ and 
$[M/L \cap W] = 2^{-1}\mathfrak{b}(q) = \mathbf{Z}$ by \REF{t31}. 
\\

As for a more precise explanation of the result, 
we take $(V,\, \p) = (\mathbf{Q}^{1}_{6},\, 1_{6})$, 
which is one of the spaces with invariants \REF{inv}. 
In fact, 
we have only to see the characteristic algebra of $1_{6}$. 
Since $Q(1_{4}) = B_{2,\, \infty}$, 
by employing \cite[Lemma 2.8(i)]{cq}, 
\begin{gather*}
A(1_{6})_{p} \cong A(1_{4})_{p} \otimes_{\mathbf{Q}_{p}} A(1_{2})_{p} \cong 
\begin{cases}
M_{8}(\mathbf{Q}_{p}) & \text{if $p \ne 2$}, \\
M_{4}(B_{2}) & \text{if $p = 2$}. 
\end{cases}
\end{gather*}
This shows $Q(1_{6}) = B_{2,\, \infty}$ as required. 
It is known that 
the genus of all maximal lattices in $(\mathbf{Q}^{1}_{6},\, 1_{6})$ consists of a single class 
with respect to both $SO^{1_{6}}$ and $O^{1_{6}}$; 
see \cite[{\S}12.12]{04}, 
for example. 
Now, 
with a maximal lattice $L$ as above 
\cite[Theorem 7.5]{Y} shows that for a squarefree positive integer $q$, 
\begin{gather*}
\begin{cases}
L[q,\, 2^{-1}\mathbf{Z}] \ne \emptyset & \text{for any $q$}, \\
L[q,\, \mathbf{Z}] \ne \emptyset & \text{if $q$ is even or $q \equiv 3 \pmod{4}$}. 
\end{cases}
\end{gather*}
Combining these facts with our results, 
we have 
\begin{prop} \label{6}
Let $(V,\, \p)$ be a quadratic space over $\mathbf{Q}$ with invariants 
$\{6,\, \mathbf{Q}(\sqrt{-1}),\, B_{2,\, \infty},\, 6\}$. 
Let $L$ be a $\mathbf{Z}$-maximal lattice in $V$ with respect to $\p$. 
Then $L[q,\, 2^{-1}\mathbf{Z}] \ne \emptyset$ for every squarefree positive integer $q$ and 
$L[q,\, \mathbf{Z}] \ne \emptyset$ for every squarefree positive integer $q$ 
such that $q$ is even or $q \equiv 3 \pmod{4}$. 
Moreover, 
suppose $q$ is a prime number. 
Pick $h\in L[q]$ and put $W = (Fh)^{\perp}$ and $\psi = \p|_{W}$. 
Then the following assertions hold: 
\begin{enumerate}
\item The invariants of $(W,\, \psi)$ are given by \REF{a1} or \REF{a2}, 
	and the discriminant ideal of $\psi$ is given by \REF{discr}. 
\item If $q \equiv 1 \pmod{4}$, 
	then $L \cap W$ is $\mathbf{Z}$-maximal in $W$ with respect to $\psi$. 
\item If $q \equiv 3 \pmod{4}$ or $q = 2$, 
	then 
	\begin{gather*}
	[M/L \cap W] = 
	\begin{cases}
	2\mathbf{Z} & \text{if $h \in L[q,\, 2^{-1}\mathbf{Z}]$}, \\ 
	\mathbf{Z} & \text{if $h \in L[q,\, \mathbf{Z}]$}, 
	\end{cases} 
	\end{gather*}
	where $M$ is a $\mathbf{Z}$-maximal lattice in $W$ with respect to $\psi$. 
	In particular, 
	$L \cap W$ is $\mathbf{Z}$-maximal in $W$ with respect to $\psi$ if $h \in L[q,\, \mathbf{Z}]$. 
\end{enumerate}
\end{prop}
\begin{prop}
In the same setting as in \PROP{6} assume that $q$ is a prime number such that 
$q \equiv 3 \pmod{4}$ and $h \in L[q,\, \mathbf{Z}]$. 
Then $[D(L \cap W):O^{\psi}_{\mathbf{A}} \cap D(L)] = [C(L \cap W):SO^{\psi}_{\mathbf{A}} \cap C(L)] = 2$ 
and $O^{\psi}\varepsilon D(L \cap W) = O^{\psi}\varepsilon (O^{\psi}_{\mathbf{A}} \cap D(L))$ 
for every $\varepsilon \in O^{\psi}_{\mathbf{A}}$. 
Consequently, 
the class number of the genus of maximal lattices in $(W,\, \psi)$ is given by 
$\#\{L[q,\, \mathbf{Z}]/\Gamma(L)\}$ for every prime number $q$ such that $q \equiv 3 \pmod{4}$. 
Here we identify $O^{\psi}(W)$ with $\{\gamma \in O^{\p}(V) \mid h\gamma = h\}$. 
\end{prop}
\begin{proof} 
We put $H = SO^{\psi}(W)$ and $H^{\bullet} = O^{\psi}(W)$. 
By \PROP{6} our assumptions imply that $L \cap W$ is $\mathbf{Z}$-maximal in $W$ with respect to $\psi$. 
Let $p$ be a prime number. 
If $p \ne q$, 
then \cite[Proposition 11.12(iv)\ and\ (v)]{04} are applicable to $(V,\, \p)_{p},\ L_{p},$ and $h$; 
hence we have $[D(L \cap W):H^{\bullet} \cap D(L)]_{p} = [C(L \cap W):H \cap C(L)]_{p} = 1$. 
If $p = q$, 
let $(Z_{q},\, \p_{q})$ be a core subspace of $V_{q}$ 
and $N_{q} (= L_{q} \cap Z_{q})$ the maximal lattice in $Z_{q}$ as in \REF{w2}. 
Then this space is isomorphic to $(\q_{q}(\sqrt{-1}),\, \kappa_{q})$ 
because the characteristic algebra $Q(\p_{q})$ is $\{\q_{q}(\sqrt{-1}),\, 1\} = M_{2}(\q_{q})$. 
Hence, 
by \cite[Theorem 8.6(vi)\ and\ Proposition 8.8(ii)]{04}, 
$A^{+}(L)_{q}$ is a maximal order in $A^{+}(V)_{q}$ such that 
$G^{+}(V)_{q} \cap A^{+}(L)_{q}^{\times} = J(V)_{q}$, 
where $J(V)_{q}$ is defined by \REF{j} with $L_{q}$. 
Moreover $\tau(J(V)_{q}) = C(L)_{q}$ by \cite[Theorem 8.9(i)]{04}. 
Similarly for $(W,\, \psi)_{q}$, 
$A^{+}(L \cap W)_{q}$ is a maximal order in $A^{+}(W)_{q}$ such that 
$G^{+}(W)_{q} \cap A^{+}(L \cap W)_{q}^{\times} = J(W)_{q}$, 
where $J(W)_{q}$ is defined by \REF{j} with the maximal lattice $(L \cap W)_{q}$. 
This can be verified in the same way as in \cite[Theorem 8.6(ii)\ and\ (vi)]{04} 
with a unique maximal order in the even Clifford algebra $Q(\psi)_{q}$ of a $3$-dimensional core space of $W_{q}$. 
Moreover $[C(L \cap W)_{q}:\tau(J(W)_{q})] = 2$ by \cite[Theorem 1.8(iii)]{06a}. 
Thus we find that $J(W)_{q} = G^{+}(W)_{q} \cap J(V)_{q}$, 
whence $[C(L \cap W)_{q}:H_{q} \cap C(L)_{q}] = 2$. 
Now, 
$Z_{q} \cap W_{q} \ne \emptyset$. 
Since this is anisotropic, 
$N_{q} \cap W_{q}$ is maximal in $Z_{q} \cap W_{q}$. 
By \cite[Lemma 6.8]{04} there exists $\gamma_{0} \in D(Z_{q} \cap W_{q})$ such that $\det(\gamma_{0}) = -1$. 
Because $\p_{q}$ is nondegenerate on $Z_{q} \cap W_{q}$, 
so is on $(Z_{q} \cap W_{q})^{\perp}$. 
We then define $\gamma \in O^{\p}(V)_{q}$ by 
$x\gamma = x\gamma_{0}$ if $x \in Z_{q} \cap W_{q}$ and $x\gamma = x$ if $x \in (Z_{q} \cap W_{q})^{\perp}$. 
Observe $\det(\gamma) = -1$ and $L_{q}\gamma = L_{q}$. 
Thus, 
in view of $h \in (Z_{q} \cap W_{q})^{\perp}$, 
we have $\gamma \in H^{\bullet}_{q} \cap D(L)_{q}$ such that $\det(\gamma) = -1$. 
This combined with the above index concludes $[D(L \cap W)_{q}:H^{\bullet}_{q} \cap D(L)_{q}] = 2$, 
which proves the first assertion. 

To prove the second assertion, 
let $\varepsilon$ be an arbitrary element of $H_{\mathbf{A}}^{\bullet}$. 
Put $\Lambda = L\varepsilon^{-1}$, 
which is a $\mathbf{Z}$-maximal lattice in $(V,\, \p)$. 
Then we have $\p(h,\, \Lambda) = \p(h,\, L) = \mathbf{Z}$. 
Let us show that 
$\tau(h) \in H^{\bullet}$ satisfies 
$(\Lambda \cap W)\tau(h) = \Lambda \cap W$ and $\Lambda \tau(h) \ne \Lambda$. 
For any prime $p \ne q$, 
we have $\p[h] \in \mathbf{Z}_{p}^{\times}$ and $2\p(h,\, \Lambda)_{p} \subset \mathbf{Z}_{p}$. 
Hence $h$ belongs to $\Lambda_{p}$ and it is invertible in the order $A(\Lambda)_{p}$. 
Since $A(\Lambda)_{p} \cap V_{p} = \Lambda_{p}$ by \cite[Lemma 8.4(iii)]{04}, 
we have $\Lambda_{p}\tau(h) = h^{-1}A(\Lambda_{p})h \cap V_{p} = \Lambda_{p}$. 
As for the prime $q$, 
we take a Witt decomposition as in \REF{w2}; 
\begin{gather*}
V_{q} = Z_{q} \oplus \sum_{i=1}^{2}(\q_{q}e_{i} + \q_{q}f_{i}),\quad 
\Lambda_{q} = N_{q} + \sum_{i=1}^{2}(\mathbf{Z}_{q}e_{i} + \mathbf{Z}_{q}f_{i}), 
\end{gather*}
where $N_{q}$ is a unique $\mathbf{Z}_{q}$-maximal lattice in a core subspace $Z_{q}$. 
Put $k = qe_{1} + f_{1}$; 
then $\p[k] = \p[h] = q$ and $\p(k,\, \Lambda_{q}) = \p(h,\, \Lambda_{q})$. 
Thus by \cite[Theorem 1.3]{06a}, 
$h\alpha = k$ and $\Lambda_{q}\alpha = \Lambda_{q}$ with some $\alpha \in O^{\p}_{q}$. 
Moreover $\tau(h\alpha) = \alpha^{-1}\tau(h)\alpha$ by \cite[Lemma 3.8(ii)]{04}. 
Then $\alpha$ gives an isomorphism of $W_{q}$ onto $W_{q}^{\prime} = (\q_{q}k)^{\perp}$ 
such that $(\Lambda_{q} \cap W_{q})\alpha = \Lambda_{q} \cap W_{q}^{\prime}$ 
and $\alpha^{-1}\tau(h)\alpha = \tau(k)$. 
Putting $U = Z_{q} + (\q_{q}e_{2} + \q_{q}f_{2})$, 
we have $\Lambda_{q} \cap W_{q}^{\prime} = (\Lambda_{q} \cap U) + \mathbf{Z}_{q}(qe_{1}-f_{1})$. 
It can be seen by \cite[Lemma 3.10]{04} that 
\begin{gather*}
(\Lambda_{q} \cap W_{q}^{\prime})\tau(k) 
	= \{-x - a(qe_{1} - f_{1}) \mid x \in \Lambda_{q} \cap U,\ a \in \mathbf{Z}_{q}\} 
	= \Lambda_{q} \cap W_{q}^{\prime}, \\ 
\Lambda_{q}\tau(k) 
	= \{-x + qae_{1} + q^{-1}bf_{1} \mid x \in \Lambda_{q} \cap U,\ a,\, b \in \mathbf{Z}_{q}\} 
	\ne \Lambda_{q}. 
\end{gather*}
Thus we obtain 
$(\Lambda_{q} \cap W_{q})\tau(h) = \Lambda_{q} \cap W_{q}$ and $\Lambda_{q}\tau(h) \ne \Lambda_{q}$. 
This gives the desired fact. 
Now, 
$[D(L \cap W):H^{\bullet}_{\mathbf{A}} \cap D(L)] = 2$ as seen above. 
Observing $D(\Lambda \cap W) = \varepsilon D(L \cap W)\varepsilon^{-1}$ and 
$H^{\bullet}_{\mathbf{A}} \cap D(\Lambda) = \varepsilon (H^{\bullet}_{\mathbf{A}} \cap D(L))\varepsilon^{-1}$, 
we have then 
$D(\Lambda \cap W) = (H^{\bullet}_{\mathbf{A}} \cap D(\Lambda)) \sqcup 
	\tau(h)(H^{\bullet}_{\mathbf{A}} \cap D(\Lambda))$. 
Since $\tau(h) \in H^{\bullet}$, 
we obtain 
$H^{\bullet}\varepsilon D(L \cap W) = 
	H^{\bullet}\varepsilon (H^{\bullet}_{\mathbf{A}} \cap D(L))$, 
and hence 
$\#\{H^{\bullet}\setminus H^{\bullet}_{\mathbf{A}}/D(L \cap W)\} = 
	\#\{H^{\bullet}\setminus H^{\bullet}_{\mathbf{A}}/(H^{\bullet}_{\mathbf{A}} \cap D(L))\}$. 
Therefore we have the last assertion by the class number formula in \REF{116}. 
\end{proof}

%

College of Science and Engineering 

Ritsumeikan University 

Kusatsu,\ Shiga 525-8577 

Japan 

murata31@pl.ritsumei.ac.jp 

\end{document}